\documentclass[12pt]{amsart}

\usepackage{amssymb, amsmath, amsthm}
\usepackage[margin=1in]{geometry}
\usepackage[dvipdfmx]{graphicx,xcolor} % Remove this package, when submitting to arXiv.
\usepackage{tikz}

\allowdisplaybreaks

\numberwithin{equation}{section}

\theoremstyle{plain}
\newtheorem{thm}{Theorem}[section]
\newtheorem{prop}[thm]{Proposition}
\newtheorem{lem}[thm]{Lemma}

\newtheorem*{referthmA}{Theorem A}
\newtheorem*{referthmB}{Theorem B}

\theoremstyle{definition}
\newtheorem{defn}[thm]{Definition}

\newtheorem{notation}[thm]{Notation}

\newcommand{\ichi}{\mathbf{1}}

\newcommand{\N}{\mathbb{N}}
\newcommand{\R}{\mathbb{R}}

\newcommand{\Z}{\mathbb{Z}}

\newcommand{\calM}{\mathcal{M}}
\newcommand{\calP}{\mathcal{P}}
\newcommand{\calS}{\mathcal{S}}
\newcommand{\dotS}{\Dot{S}}
\newcommand{\supp}{\mathrm{supp}\, }

\newcommand{\card}{\mathrm{card}\, }
\newcommand{\dist}{\, \mathrm{dist}\, }
\newcommand{\diam}{\, \mathrm{diam}\, }

\newcommand{\RomI}{\mathrm{I}}
\newcommand{\II}{\mathrm{I}\hspace{-0.5pt}\mathrm{I}}
\newcommand{\III}{\mathrm{I}\hspace{-0.5pt}\mathrm{I}\hspace{-0.5pt}\mathrm{I}}
\begin{document}
%%%%%%%%%%%%%%%%%%%%%%%%%%%%%%%%%

%%%============================
\title[Bilinear wave operators]
{Estimates for some bilinear wave operators} 

\author[T. Kato]{Tomoya Kato}
\author[A. Miyachi]{Akihiko Miyachi}
\author[N. Tomita]{Naohito Tomita}

\address[T. Kato]
{Division of Pure and Applied Science, 
Faculty of Science and Technology, Gunma University, 
Kiryu, Gunma 376-8515, Japan}

\address[A. Miyachi]
{Department of Mathematics, 
Tokyo Woman's Christian University, 
Zempukuji, Suginami-ku, Tokyo 167-8585, Japan}

\address[N. Tomita]
{Department of Mathematics, 
Graduate School of Science, Osaka University, 
Toyonaka, Osaka 560-0043, Japan}

\email[T. Kato]{t.katou@gunma-u.ac.jp}
\email[A. Miyachi]{miyachi@lab.twcu.ac.jp}
\email[N. Tomita]{tomita@math.sci.osaka-u.ac.jp}

\date{\today}


\thanks
{This work was supported by JSPS KAKENHI, 
Grant Numbers 
20K14339 (Kato), 
20H01815 (Miyachi), and 
20K03700 (Tomita).}

%(Singular and oscillatory integrals (Calder\'on-Zygmund, etc.))

%%%================================
\begin{abstract}
We consider some bilinear Fourier multiplier operators
and give a bilinear version of
Seeger, Sogge, and Stein's result for Fourier integral operators.
Our results improve, for the case of Fourier multiplier operators,
Rodr\'iguez-L\'opez, Rule, and Staubach's result for
bilinear Fourier integral operators.
The sharpness of the results is also considered.
\end{abstract}

%%%=====================================
\maketitle
%%%=====================================

%%%%%%%%%%%%%%%%%%%%%%%%%%%%%%%%%
\section{Introduction}\label{intro}
%%%%%%%%%%%%%%%%%%%%%%%%%%%%%%%%%

The solution to the wave equation 
$\partial_t^2 u = \triangle u$ with the initial data 
$u(0, x)=f(x)$ and $u_t (0,x)=g(x)$ is given by 
\[
u(t,x) 
=
\frac{1}{(2\pi)^{n}}
\int_{\R^n}
e^{ix\cdot \xi}
\cos (t |\xi|)\, \widehat{f}(\xi)\, 
d\xi
+
\frac{1}{(2\pi)^{n}}
\int_{\R^n}
e^{ix\cdot \xi}\, 
\frac{\sin (t |\xi|)}{|\xi|}\, 
\widehat{g}(\xi)\, 
d\xi, 
\]
where $\widehat{f}$ denotes the Fourier transform of 
$f$ (for the definition of  Fourier transform, see 
Notation \ref{notation} below). 
Several basic properties of the mapping 
$(f,g) \mapsto u(t, \cdot)$ are derived from the estimate 
of the operator 
\begin{equation}\label{linear-wave}
Tf (x)= \frac{1}{(2\pi)^{n}}
\int_{\R^n}
e^{ix \cdot \xi}\, e^{i |\xi|} \, (1+|\xi|^2)^{m/2}
\, \widehat{f}(\xi)\, 
d\xi. 
\end{equation}
The purpose of this paper is to consider bilinear versions of this operator.

We begin with the definition of linear Fourier multiplier operators.

For $\theta \in L^{\infty} (\R^n)$, 
%$\theta =\theta (\xi)$ be a measurable function on $\R^n$ 
%that is locally integrable and of at most polynomial growth as $|\xi| \to \infty$. 
the operator $\theta (D)$ is defined by 
\[
\theta (D) f (x) 
=
\frac{1}{(2\pi)^{n}}
\int_{\R^n}
e^{i x \cdot \xi}\, 
\theta (\xi) 
\widehat{f}(\xi) 
\, d\xi, 
\quad 
x \in \R^n, 
\]
for $f$ in the Schwartz class $\calS (\R^n)$. 
%The function $\theta$ is called the multiplier.  
If $X$ and $Y$ are function spaces on $\R^n$ 
equipped with quasi-norms or seminorms 
$\|\cdot \|_{X}$ and $\|\cdot \|_{Y}$, respectively, 
and if there exists a constant $A$ 
such that 
\[
\| 
\theta (D) f \|_{Y} 
\le A \|  f \|_{X} 
\quad \text{for all}\;\; 
f \in X \cap \calS, 
 \]
then we say that 
$\theta$ is a 
{\it Fourier multiplier}\/  
for $X\to Y$ 
and write $\theta \in \calM (X\to Y)$. 
(Sometimes we write 
$\theta (\xi) \in \calM (X\to Y)$ to mean 
$\theta (\cdot) \in \calM (X\to Y)$.)  
The minimum of $A$ that satisfies the above inequality 
is denoted by $\|\theta\|_{\calM (X\to Y)}$.

Throughout this paper, $H^p$, $0<p\le \infty$,  
denotes the Hardy space and $BMO$ denotes the 
space of bounded mean oscillation. 
We use the convention that $H^p = L^p$ if $1<p\le \infty$. 
For $H^p$ and $BMO$, see, {\it e.g.,}\/ \cite[Chapters III and IV]{S}.

We recall classical results about the operator 
\eqref{linear-wave} and its generalizations. 
We use the following notation. 

%%%=============================
\begin{defn}\label{def-P} 
We write $\calP = \calP (\R^n)$ to denote 
the set of all functions on $\R^n$ that are 
real-valued, homogeneous of degree $1$, and 
$C^{\infty}$ away from the origin.
\end{defn}
%%%=============================

The following theorem is due to Seeger, Sogge, and Stein \cite{SSS}. 

%%%=============================
\begin{referthmA}[{Seeger--Sogge--Stein \cite{SSS}}] 
If $\phi \in \calP (\R^n)$, 
$1\le p\le \infty$, and $m=  -(n-1) |1/p-1/2|$, 
then 
\begin{equation*}
e^{i \phi (\xi) } \big( 1+  |\xi|^2\big)^{m/2} 
\in 
\begin{cases}
{\calM (H^p \to H^p)} & \text{when $1\le p<\infty$, } \\
{\calM (BMO \to BMO)} & \text{when $p=\infty$. }
\end{cases}
\end{equation*}
\end{referthmA}
%%%=============================

In fact, this theorem is not given in \cite{SSS} 
in exactly the same form as above; 
the result given in \cite{SSS} is restricted to local estimate. 
However, Theorem A can be proved by a slight modification 
of the argument of \cite{SSS}. 
Or one can appeal to the general results given by 
Ruzhansky and Sugimoto \cite[Theorems 1.2 and 2.2]{RS}.

It is known that the number $-(n-1)|1/p-1/2|$ 
given in Theorem A is optimal. 
In fact, for the typical case $\phi (\xi)=|\xi|$, 
the following theorem holds.

%%%=============================
\begin{referthmB}
If $1\le p\le \infty$ and if 
\begin{equation*}
e^{i |\xi| } \big( 1+  |\xi|^2\big)^{m/2} 
\in 
\begin{cases}
{\calM (H^p \to H^p)} & \text{when $1\le p<\infty$, } \\
{\calM (BMO \to BMO)} & \text{when $p=\infty$, }
\end{cases}
\end{equation*}
then $m\le -(n-1) |1/p-1/2|$. 
\end{referthmB}
%%%=============================

For a proof of this theorem, 
see \cite[Theorem 1]{M-wave} or \cite[Chapter IX, 6.13]{S}.

The purpose of the present 
paper is to consider bilinear versions of Theorems A and B.

We recall the definition of bilinear Fourier multiplier operators. 
For a bounded measurable 
function $\sigma = \sigma (\xi, \eta)$ on $\R^n\times \R^n$,
the bilinear operator
$T_{\sigma}$ is defined by
\[
T_{\sigma}(f, g)(x)
=\frac{1}{(2\pi)^{2n}}
\iint_{\R^n\times \R^n} 
e^{i x \cdot(\xi + \eta)}
\sigma (\xi, \eta) \,
\widehat{f}(\xi)\, 
\widehat{g}(\eta)
\, d\xi d\eta, 
\quad 
x \in \R^n, 
\]
for $f, g \in \calS(\R^n)$. 
If $X, Y$, and $Z$ are 
function spaces on $\R^n$ 
equipped with quasi-norms or seminorms 
$\|\cdot \|_{X}$, $\|\cdot \|_{Y}$, and $\|\cdot \|_{Z}$, 
respectively, and if there exists a constant $A$ such that 
\begin{equation*}
\|T_{\sigma}(f, g)\|_{Z}
\le A \|f\|_{X}\, \|g\|_{Y} 
\quad 
\text{for all}
\;\;
f \in X \cap \calS
\;\; 
\text{and all}\;\; 
g \in Y \cap \calS, 
\end{equation*}
then 
we say that 
$\sigma$ is a {\it bilinear Fourier multiplier}\/ for 
$X \times Y$ to $Z$ and 
write 
$\sigma \in \calM (X \times Y \to Z)$.  
(Sometimes we write 
$\theta (\xi, \eta) \in \calM (X\times Y \to Z)$ to mean 
$\theta (\cdot, \cdot) \in \calM (X\times Y \to Z)$.)  
The smallest constant $A$ that satisfies the above inequality 
is denoted by
$\|\sigma\|_{\calM (X\times Y \to Z)}$.

We shall consider the bilinear Fourier multiplier of the form 
\[
%\tau (\xi, \eta)= 
e^{i (\phi_1 (\xi)+ \phi_2 (\eta))} 
\sigma (\xi, \eta), 
\quad \phi_1, \phi_2 \in \calP (\R^n), 
\quad  \sigma \in S^{m}_{1,0} (\R^{2n}), 
\]
where the class $S^{m}_{1,0} (\R^{2n})$ is defined as 
follows. 
%%%=============================
\begin{defn}\label{def-Sm10} 
For $m\in \R$, 
the class $S^{m}_{1,0} (\R^{2n})$ is defined to be the set of all 
$C^{\infty}$ functions $\sigma = \sigma (\xi, \eta)$ 
on $\R^{2n}$ 
that satisfy the estimate 
\begin{equation*}
|\partial^{\alpha}_{\xi} 
\partial^{\beta}_{\eta} \sigma (\xi, \eta)| 
\le 
C_{\alpha} \left( 1+ |\xi| + |\eta|\right)^{m-|\alpha|-|\beta|}
\end{equation*}
for all multi-indices $\alpha, \beta$. 
\end{defn}
%%%=============================

In the theory of bilinear Fourier multipliers, 
a classical method is known that allows us 
to write multiplier $\sigma \in S^{m}_{1,0}(\R^{2n})$ 
as a sum of multipliers of the product form 
$\theta_{1} (\xi) \theta_{2} (\eta)$. 
Using this method, we can deduce 
the following theorem from Theorem A.

%%%=============================
\begin{thm}\label{th-C2} 
Let  
$n\ge 2$,  
$1\le p, q\le \infty$, and $1/p+1/q=1/r$.  
Assume 
$\phi_1, \phi_2 \in \calP (\R^n)$ and 
$\sigma \in S^{m}_{1,0}(\R^{2n})$ with 
$m =     
-(n-1) \big( |{1}/{p}- {1}/{2}|+|{1}/{q}- {1}/{2}| \big) $.
Then 
$e^{i (\phi_1 (\xi)+ \phi_2 (\eta))} 
\sigma (\xi, \eta)
\in 
\calM (H^p \times H^q \to L^r)$, 
where $L^r$ should be replaced by $BMO$ when $r=\infty$. 
\end{thm}
%%%=============================

In fact, 
Rodr\'iguez-L\'opez--Rule--Staubach \cite{RRS} 
considered more general operators, bilinear 
Fourier integral operators, and proved a theorem that almost covers 
Theorem \ref{th-C2}.  
The statement of the theorem of \cite{RRS} is, however, 
restricted to local estimate. 
We shall give a full proof of Theorem \ref{th-C2}  
in a succeeding section, Section \ref{C}.

The main purpose of the present paper is to 
show that the number 
$m= -(n-1) \big( |{1}/{p}- {1}/{2}|+|{1}/{q}- {1}/{2}| \big) $ 
in Theorem \ref{th-C2} can be improved 
%in the case 
%$1\le p\le 2 \le q \le \infty$ or 
%$1\le q\le 2 \le p \le \infty$
and show that the improved $m$ is 
optimal at least for certain $(p,q)$.

The following is the first main theorem of this paper.

%%%=============================
\begin{thm}\label{th-E1D4} 
Let  
$n\ge 2$,  
$1\le p, q\le \infty$, and $1/p+1/q=1/r$.  
Assume 
$\phi_1, \phi_2 \in \calP (\R^n)$ and 
$\sigma \in S^{m}_{1,0}(\R^{2n})$ with 
$m= m_1 (p,q)$, where  
\begin{equation*}
m_1 (p,q)
=
\begin{cases}
{-(n-1) \big( 
|\frac{1}{p}- \frac{1}{2}|+|\frac{1}{q}- \frac{1}{2}| \big) }
& 
\text{if $1\le p,q \le 2$ or if $2\le p,q \le \infty$,} 
\\
{-\big( 
\frac{1}{p} - \frac{1}{2} \big) 
- (n-1) \big( \frac{1}{2} - \frac{1}{q} \big) }
& 
\text{if $1\le p\le 2\le q \le \infty$ and $\frac{1}{p}+\frac{1}{q} \le 1$,}
\\
{-(n-1)
\big( \frac{1}{p} - \frac{1}{2} \big) 
- \big( \frac{1}{2} - \frac{1}{q} \big) }
& 
\text{if $1\le p\le 2\le q \le \infty$ and $\frac{1}{p}+\frac{1}{q}\ge 1$,}
\\
{-(n-1)
\big( 
\frac{1}{2} - \frac{1}{p} \big) 
- \big( \frac{1}{q} - \frac{1}{2} \big) }
& 
\text{if $1\le q\le 2\le p \le \infty$ and $\frac{1}{p}+\frac{1}{q} \le 1$,}
\\
{-\big( \frac{1}{2} - \frac{1}{p} \big) 
- (n-1) \big( \frac{1}{q} - \frac{1}{2} \big) }
& 
\text{if $1\le q\le 2\le p \le \infty$ and $\frac{1}{p}+\frac{1}{q}\ge 1$.}
\end{cases}
\end{equation*}
Then 
%\eqref{eq-nec1} holds for 
%$m\le m_1 (p,q)$. 
$e^{i (\phi_1 (\xi)+\phi_2 (\eta))} \sigma (\xi, \eta)
\in 
\calM (H^p \times H^q \to L^r)$,   
where $L^r$ should be replaced by $BMO$ when $r=\infty$. 
\end{thm}
%%%=============================

Compare the claims of Theorems \ref{th-C2} and \ref{th-E1D4}. 
They are the same 
in the regions $1\le p,q \le 2$ and $2\le p,q \le \infty$, 
but different outside of these regions. 
In the typical case $(p,q)=(1, \infty)$,  
Theorem \ref{th-C2} asserts that the multiplier 
$e^{i (\phi_1 (\xi)+\phi_2 (\eta))} \sigma (\xi, \eta)$ 
belongs to $\calM (H^1\times L^{\infty}\to L^1)$ 
if $\sigma \in S^{-(n-1)}_{1,0}(\R^{2n})$, 
whereas Theorem \ref{th-E1D4} asserts that the same 
holds 
if $\sigma \in S^{-\frac{n}{2}}_{1,0}(\R^{2n})$. 
The latter is stronger if $n\ge 3$. 
To be precise, observe that 
$m_1 (p,q) > 
-(n-1) \big( |{1}/{p}- {1}/{2}|+|{1}/{q}- {1}/{2}| \big) $ 
if $n\ge 3$ and 
$1\le p < 2 < q\le \infty$ 
or 
$1\le q < 2 < p\le \infty$. 
Thus Theorem \ref{th-E1D4} is an improvement of Theorem \ref{th-C2} 
for these $n, p, q$.

In order to show that 
the number $m_1 (p,q)$ is in fact optimal for some $(p,q)$, 
we consider the special case $\phi_1 (\xi)=\phi_2 (\xi)=|\xi|$.  
We write 
\begin{equation}\label{def-Xr}
X_r = 
\begin{cases}
{L^r} & \text{if $0<r<\infty$, } \\
{BMO} & \text{if $r=\infty$.}
\end{cases}
\end{equation}
For $p,q \in [1, \infty]$ given, set $1/r=1/p+1/q$ 
and we consider 
necessary condition on 
$m\in \R$ that allows the assertion  
\begin{equation}\label{eq-nec1}
e^{i (|\xi|+ |\eta|)} 
\sigma (\xi, \eta) 
\in \calM (H^p \times H^q \to X_r) 
\;\; 
\text{for all}
\;\; 
\sigma \in S^{m}_{1,0} (\R^{2n}). 
\end{equation}

The following is the second main theorem of this paper.

%%%=============================
\begin{thm}\label{th-E1E3} 
Let $n \ge 2$. 

$(1)$ Let $1 \le p, q \le 2$ or $2\le p, q \le \infty$. 
Then 
$m \in \R$ satisfies \eqref{eq-nec1} 
only if 
$m \le - (n-1) \big( |1/p - 1/2| + |1/q - 1/2| \big)$. 

$(2)$ Let $1\le p \le 2\le q \le \infty$ 
or $1\le q \le 2\le p \le \infty$ 
and assume 
$1/p+1/q=1$. 
Then $m \in \R$ satisfies 
\eqref{eq-nec1} 
only if $m \le -n |1/p -1/2|$. 
\end{thm}
%%%=============================

This theorem implies that the number 
$m_1 (p,q)$ of Theorem \ref{th-E1D4} 
is optimal 
for $p,q$ in the range given in (1) and (2) of Theorem \ref{th-E1E3}. 
The present authors do not know whether 
$m_1 (p,q)$ is optimal for other $p,q$.

The contents of the rest of the paper are 
as follows. 
In Section \ref{flag}, we collect 
some propositions concerning flag paraproduct, 
which we will use in the proof of Theorem \ref{th-C2}. 
In order not to interrupt the stream of argument, 
we shall postpone 
rather long proofs of those propositions 
to Section \ref{G}. 
In Sections \ref{C}, \ref{D}, and \ref{E}, 
we prove Theorems \ref{th-C2}, \ref{th-E1D4}, and 
\ref{th-E1E3}, respectively. 
The last section, Section \ref{G}, is devoted to the proofs of 
the propositions stated in Section \ref{flag}.

We end this section by introducing some notations 
used throughout this paper.

%%======================
\begin{notation}\label{notation}
The Fourier transform and the inverse Fourier transform 
on $\R^d$ are defined by 
\begin{align*}
&
\widehat{f}(\xi)
=
\int_{\R^d}
e^{-i \xi \cdot x}
f(x)\, dx, 
\\
&
(g)^{\vee}(x)
=
\frac{1}{(2\pi)^d} 
\int_{\R^d}
e^{i \xi \cdot x}
g(\xi)\, d\xi. 
\end{align*}
Sometimes we use rude expressions 
$\big( f(x) \big)^{\wedge}$ 
or 
$\big( g(\xi) \big)^{\vee}$ 
to denote 
$\big( f(\cdot ) \big)^{\wedge}$ 
or  
$\big( g(\cdot ) \big)^{\vee}$, respectively.

We shall repeatedly use dyadic partition of unity, which is defined as follows.   
Take a function  
$\psi \in C_{0}^{\infty}(\R^n)$ such that 
$\supp \psi \subset \{2^{-1}\le |\xi| \le 2\}$ and 
$\sum_{j=-\infty}^{\infty} \psi (2^{-j} \xi)=1$ 
for
$\xi \neq 0$.
We define functions $\zeta$ and 
$\varphi$ by 
$\zeta (\xi)
=
\sum_{j=1}^{\infty} \psi (2^{-j}\xi)$ 
and 
$\varphi (\xi)= 1- \zeta (\xi)$. 
We have 
\begin{align*}
& 
\zeta (\xi)=0 \;\; \text{if}\;\; |\xi|\le 1,
\quad 
\zeta (\xi)=1 \;\; \text{if}\;\; |\xi|\ge 2,
\\
&
\varphi (\xi)=1 \;\; \text{if}\;\; |\xi|\le 1,
\quad 
\varphi (\xi)=0 \;\; \text{if}\;\; |\xi|\ge 2,
\\
&
\sum_{j=-\infty}^{k} \psi (2^{-j} \xi) = \varphi (2^{-k} \xi),  
\quad \xi \neq 0, \;\; k\in \Z. 
\end{align*}
Notice, however, that we will also use the letters $\psi$, $\zeta$, $\varphi$ 
in a meaning different from the above.

For a smooth function $\theta$ on $\R^d$ and for a nonnegative integer   
$N$, we write 
$\|\theta\|_{C^N}=
\max_{|\alpha|\le N} 
\sup_{\xi} \big| \partial_{\xi}^{\alpha} \theta (\xi) \big|$. 

The letter $n$ denotes the dimension of the Euclidean space 
that we consider.  
Unless further restrictions are explicitly made, 
$n$ is an arbitrary positive integer. 
\end{notation}
%%========================

%%%%%%%%%%%%%%%%%%%%%%%%%%%%%%%%%

%%%%%%%%%%%%%%%%%%%%%%%%%%%%%%%%%
\section{Some results from bilinear flag paraproducts}\label{flag}
%%%%%%%%%%%%%%%%%%%%%%%%%%%%%%%%%

In this section, we give some results 
for the bilinear Fourier multipliers of the form 
\[
a_0 (\xi, \eta) a_1 (\xi) a_2 (\eta).
\]
This kind of multipliers with $a_0, a_1, a_2$ 
being the $0$-th order multipliers ({\it i.e.}\/ the ones 
that generalize the homogeneous functions of degree $0$) 
are considered by Muscalu \cite{Muscalu1, Muscalu2} 
and Muscalu--Schlag \cite[Chapter 8]{MuscaluSchlag}, 
where their mapping properties between $L^p$ spaces are given. 
In this section, we consider the case 
where $a_0, a_1, a_2$ are non-zero order multipliers 
and give estimates including $H^p$ and $BMO$. 
The results of this section will be used to prove 
Theorem \ref{th-C2}.

%%%=============================
\begin{defn}\label{def-dotSm10} 
For $m\in \R$ and $d\in \N$, 
the class $\dot{S}^{m}_{1,0} (\R^d)$ is defined to be the set of all 
$C^{\infty}$ functions $\theta$ on $\R^d \setminus \{0\}$ 
such that 
\begin{equation*}
|\partial^{\alpha}_{\xi} \theta (\xi)| 
\le 
C_{\alpha} |\xi|^{m-|\alpha|}
\end{equation*}
for all multi-indices $\alpha$. 
\end{defn}
%%%=============================

We first recall a classical result about the bilinear 
Fourier multipliers in the class $\dot{S}^{0}_{1,0}(\R^{2n})$. 
The following proposition was established by the works of 
Coifman--Meyer \cite{CM1, CM2, CM3}, 
Kenig--Stein \cite{KS}, 
Grafakos--Torres \cite{GT}, 
and 
Grafakos--Kalton \cite{GK1}.

%%%=============================
\begin{prop}\label{prop-CM} 
If $\sigma \in \dot{S}^{0}_{1,0}(\R^{2n})$, 
then $\sigma \in \calM (H^p \times H^q \to L^r)$ 
for $0<p, q\le \infty$ and $1/p+1/q=1/r >0$,  
and also $\sigma \in \calM (L^{\infty} \times L^{\infty} \to BMO)$.  
\end{prop}
%%%=============================

Proofs of the following two propositions will be given in Section \ref{G}.

%%%=============================
\begin{prop}\label{prop-a0a1a2} 
Let $m_1, m_2 \le 0$, $m=m_1 + m_2$, 
$a_0 \in \dot{S}^{m}_{1,0}(\R^{2n})$, 
$a_1 \in \dot{S}^{-m_1}_{1,0}(\R^{n})$, 
$a_2 \in \dot{S}^{-m_2}_{1,0}(\R^{n})$, 
and let 
$
\sigma (\xi, \eta)
=a_0 (\xi, \eta) 
a_1 (\xi) a_2 (\eta)$.  
Then the following hold. 
\begin{itemize}
\item[$(1)$] $\sigma \in \calM (H^p \times H^q \to L^r)$ 
for $0<p, q < \infty$ and $1/p+1/q=1/r$. 
\item[$(2)$] If $m_2 <0$, then 
$\sigma \in \calM (H^p \times BMO \to L^p)$ 
for $0<p< \infty$. 
\item[$(3)$] If $m_1 <0$, then 
$\sigma \in \calM (BMO \times H^q \to L^q)$ 
for $0<q< \infty$. 
\item[$(4)$] If $m_1, m_2 <0$, then 
$\sigma \in \calM (BMO \times BMO \to BMO)$. 
\end{itemize}
\end{prop}
%%%=============================

%%%=============================
\begin{prop}\label{prop-a0a1} 
Let $m_1 \le 0$, 
$a_0 \in \dot{S}^{m_1}_{1,0}(\R^{2n})$, 
$a_1 \in \dot{S}^{-m_1}_{1,0}(\R^{n})$, 
and let 
$\tau (\xi, \eta)
=a_0 (\xi, \eta) 
a_1 (\xi)$.  
Then the following hold. 

$(1)$ $\tau \in \calM (H^p \times L^{\infty} \to L^p)$ 
for $0<p < \infty$. 

$(2)$  
If $m_1 <0$, then 
$\tau \in \calM (BMO \times L^{\infty} \to BMO)$. 
\end{prop}
%%%=============================

%%%%%%%%%%%%%%%%%%%%%%%%%%%%%%%%%
\section{Proof of Theorem \ref{th-C2}}
\label{C}
%%%%%%%%%%%%%%%%%%%%%%%%%%%%%%%%%

In order to prove Theorem \ref{th-C2}, we use the following lemma.

%%%=============================
\begin{lem}\label{lem-D1} 
If $\phi \in \calP (\R^n)$ 
and if $\theta \in C_0^{\infty}(\R^n)$ 
satisfy $\supp \theta \subset \{|\xi|\le 2\}$, 
then 
\[
\left\| \left( e^{i \phi (\xi)} \theta (\xi) \right)^{\vee}
\right\|_{L^1} 
\le c \|\theta\|_{C^{n+1}},
\]
 where $c=c(n, \phi)$. 
\end{lem}
%%%=============================

\begin{proof} 
Write 
\[
e^{i \phi (\xi)} \theta (\xi)
%=\phi (\xi) + \big( e^{i |\xi|} -1\big) \phi (\xi) 
=\theta (\xi) + 
\sum_{j=-\infty}^{1} 
\big( e^{i \phi (\xi)} -1\big) \theta (\xi) \psi (2^{-j}\xi), 
\]
where $\psi$ is the function given in Notation \ref{notation}. 
The inverse Fourier transform of $\theta (\xi)$ 
satisfies 
$
\big| 
(\theta)^{\vee} (x) 
\big| 
\lesssim 
\|\theta\|_{C^{n+1}} (1+|x|)^{-n-1}
$ 
and hence 
$\big\| 
( \theta)^{\vee} \big\|_{L^1} 
\lesssim 
\|\theta \|_{C^{n+1}} $. 
The function $\big(e^{i \phi (\xi) } -1\big) \theta (\xi) \psi (2^{-j}\xi)$ 
has support 
included in $\{2^{j-1}\le |\xi|\le 2^{j+1}\}$ 
and satisfies the estimate 
\[
\left| \partial_{\xi}^{\alpha} 
\big( 
(e^{i \phi (\xi)} -1 ) \theta (\xi) \psi (2^{-j}\xi) 
\big) 
\right| 
\lesssim 
\|\theta \|_{C^{n+1}} (2^j)^{1-|\alpha|}, 
\quad |\alpha|\le n+1.  
\]
From this we obtain 
\[
\left| 
\big( ( e^{i \phi (\xi)} -1 ) 
\theta (\xi) \psi (2^{-j}\xi) 
\big)^{\vee}(x) 
\right| 
\lesssim 
\|\theta\|_{C^{n+1}} 
2^{j (n+1)}\big( 1 + 2^j |x| \big)^{-n-1}
\]
and hence 
$\big\| \big( ( e^{i \phi (\xi)} -1) 
\phi (\xi) \theta (2^{-j}\xi) \big)^{\vee}
\big\|_{L^1} \lesssim 
\|\theta\|_{C^{n+1}} 2^j$. 
Taking sum over $j\le 1$, 
we obtain 
$\big \| 
\big(e^{i \phi (\xi)} -1\big) \phi (\xi) 
\big)^{\vee} \big\|_{L^1} 
\lesssim 
\|\theta\|_{C^{n+1}} $. 
\end{proof}

\begin{proof}[Proof of Theorem \ref{th-C2}]
We write 
$m_1 = -(n-1)|1/p-1/2|$, 
$m_2 = -(n-1)|1/q-1/2|$, 
and $1/p+1/q=1/r$. 
We also use the notation \eqref{def-Xr}.

Using the functions $\zeta$ and $\varphi$ 
of Notation \ref{notation}, 
we decompose $\tau $ as 
\begin{align*}
&
\tau (\xi, \eta) =
\tau_1 (\xi, \eta) + 
\tau_2 (\xi, \eta) + 
\tau_3 (\xi, \eta) +
\tau_4 (\xi, \eta), 
\\
&
\tau_1 (\xi, \eta) = 
e^{i \phi_1 (\xi)} \varphi (\xi) 
e^{i \phi_2 (\eta) } \varphi (\eta) 
\sigma (\xi, \eta), 
\\
&
\tau_2 (\xi, \eta) = 
e^{i \phi_1 (\xi)} \zeta (\xi) 
e^{i \phi_2 (\eta) } \varphi (\eta) 
\sigma (\xi, \eta), 
\\
&
\tau_3 (\xi, \eta) = 
e^{i \phi_1 (\xi)} \varphi (\xi) 
e^{i \phi_2 (\eta) } \zeta (\eta) 
\sigma (\xi, \eta), 
\\
&
\tau_4 (\xi, \eta) = 
e^{i \phi_1 (\xi)} \zeta (\xi) 
e^{i \phi_2 (\eta) } \zeta (\eta) 
\sigma (\xi, \eta). 
\end{align*}
We shall prove 
$\tau_{i} \in \calM (H^p \times H^q \to X_r)$ 
for $i=1, 2, 3, 4$.

Firstly, the multiplier $\tau_1$ is easy to handle.  
By Lemma \ref{lem-D1}, 
the inverse Fourier transform of 
$e^{i \phi_1 (\xi) } \varphi (\xi)$ is in $L^{1}(\R^n)$ 
and hence 
$
e^{i \phi_1 (\xi) } \varphi (\xi) \in \calM (H^p \to H^p)$, 
$1\le p\le \infty$. 
Similarly 
$e^{i \phi_2 (\eta)  } \varphi (\eta) \in \calM (H^q \to H^q)$, 
$1\le q \le \infty$. 
Also 
$\sigma \in \calM (H^p \times H^q \to X_{r})$ 
by Proposition \ref{prop-CM}. 
Combining these facts, 
we have 
$\tau_1 \in \calM (H^p \times H^q \to X_{r})$.

Next, consider $\tau_2$. 
We write this as 
\[
\tau_2 (\xi, \eta) = 
\sigma (\xi, \eta)
\widetilde{\zeta} (\xi) |\xi|^{-m_1}
\cdot 
e^{i \phi_1 (\xi)} \zeta (\xi) |\xi|^{m_1} \cdot 
e^{i \phi_2 (\eta) } \varphi (\eta),  
\] 
where $\widetilde{\zeta}$ is a $C^{\infty}$ function on 
$\R^n$ such that $\widetilde{\zeta}(\xi)=1$ for $|\xi|\ge 1$ 
and $\widetilde{\zeta}(\xi)=0$ for $|\xi|\le 2^{-1}$. 
As we have seen above, 
$e^{i \phi_2 (\eta) } \varphi (\eta) \in \calM (H^q \to H^q)$ 
for $1\le q\le \infty$. 
Theorem A implies  
\[
e^{i \phi_1 (\xi)} \zeta (\xi) |\xi|^{m_1} 
\in 
\begin{cases} 
{\calM (H^p \to  H^p) } & \text{if}\;\;  1\le p <\infty, \\
{\calM (BMO \to BMO) } & {\text{if}\;\; p=\infty} . 
\end{cases}
\] 
Notice that 
$\sigma \in S^{m}_{1,0}(\R^{2n}) 
\subset 
\dotS ^{m_1}_{1,0}(\R^{2n})
$ 
and 
$\widetilde{\zeta} (\xi) |\xi|^{-m_1} \in \dotS ^{-m_1}_{1,0}(\R^n)$. 
Hence Propositions \ref{prop-a0a1a2} and \ref{prop-a0a1} give 
\begin{equation}\label{sigma-zeta}
\sigma (\xi, \eta) 
\widetilde{\zeta} (\xi) |\xi|^{-m_1} 
\in 
\begin{cases}
{\calM (H^p \times H^q \to L^r) } & \text{if}\;\; 1\le p, q <\infty, 
\\
{\calM (H^p \times L^{\infty} \to L^p) } & 
\text{if}\;\; 1\le p<\infty 
\;\;\text{and}\; \; q=\infty, 
\\
{\calM (BMO \times H^q \to L^q) } & 
\text{if}
\;\; p=\infty \;\;\text{and}\; \; 1\le q<\infty, 
\\
{\calM (BMO \times L^{\infty} \to BMO) } & {\text{if}\;\; p=q=\infty} 
\end{cases}
\end{equation} 
(notice that $m_1 < 0$ if $n\ge 2$ and $p=\infty$). 
Combining these results, 
we see that $\tau_2$ belongs to the same multiplier class 
as in \eqref{sigma-zeta}, 
which {\it a fortiori}\/ implies 
$\tau_2 \in \calM (H^p \times H^q \to X_{r})$.

By symmetry, we also have 
$\tau_3 \in \calM (H^p \times H^q \to X_{r})$.

Finally, consider $\tau_4$. 
We write this as 
\[
\tau_4 (\xi, \eta) = 
\sigma (\xi, \eta)
\widetilde{\zeta} (\xi) |\xi|^{-m_1}
\widetilde{\zeta} (\eta) |\eta|^{-m_2}
\cdot 
e^{i \phi_1 (\xi)} \zeta (\xi) |\xi|^{m_1} \cdot 
e^{i \phi_2 (\eta) } \zeta (\eta) |\eta|^{m_2},  
\] 
where $\widetilde{\zeta}$ is the same as above. 
Theorem A gives 
\begin{align*}
&
e^{i \phi_1 (\xi)} \zeta (\xi) |\xi|^{m_1} 
\in 
\begin{cases}
{\calM (H^p \to H^p) } & \text{if}\;\; 1\le p <\infty, \\
{\calM (BMO \to BMO) } & \text{if}\;\; p=\infty, 
\end{cases}
\\
&
e^{i \phi_2 (\eta)} \zeta (\eta) |\eta|^{m_2} 
\in 
\begin{cases}
{\calM (H^q \to  H^q) } & \text{if}\;\; 1\le q <\infty, \\
{\calM (BMO \to BMO) } & \text{if}\;\; q=\infty. 
\end{cases}
\end{align*} 
Proposition \ref{prop-a0a1a2} gives  
\begin{equation}
\label{sigma-zeta-zeta}
\sigma (\xi, \eta) 
\widetilde{\zeta} (\xi) |\xi|^{-m_1} 
\widetilde{\zeta} (\eta) |\eta|^{-m_2}
\in 
\begin{cases}
{\calM (H^p \times H^p \to L^r) } & \text{if}\;\; 1\le p, q <\infty, 
\\
{\calM (H^p \times BMO \to L^p) } & 
\text{if}\;\; 1\le p<\infty \;\;\text{and}\; \; q=\infty, 
\\
{\calM (BMO \times H^q \to L^q) } & 
\text{if}\;\; p=\infty \;\;\text{and}\; \; 1\le q<\infty, 
\\
{\calM (BMO \times BMO \to BMO) } & \text{if}\;\; p=q=\infty
\end{cases}
\end{equation} 
(notice that 
$m_1 < 0$ if $n\ge 2$ and $p=\infty$ 
and 
that $m_2 <0$ if $n\ge 2$ and $q=\infty$). 
Now combining these results, 
we see that $\tau_4$ belongs to the same multiplier class 
as in \eqref{sigma-zeta-zeta}, 
which {\it a fortiori}\/ implies 
$\tau_4 \in \calM (H^p \times H^q \to X_{r})$. 
This completes the proof of Theorem \ref{th-C2}.
\end{proof}

%%%%%%%%%%%%%%%%%%%%%%%%%%%%%%%%%
\section{Proof of Theorem \ref{th-E1D4}}
\label{D}
%%%%%%%%%%%%%%%%%%%%%%%%%%%%%%%%%

In this section, we prove 
Theorem \ref{th-E1D4}. 
For this, the key is to prove the assertion of 
Theorem \ref{th-E1D4} 
in the special case $p=1$ and $q=\infty$, 
which we shall write here for the sake of reference.

%%%=============================
\begin{thm}\label{th-D4} 
If  
$n\ge 2$, $\phi_1, \phi_2 \in \calP (\R^n)$, 
and $\sigma \in S^{-n/2}_{1,0}(\R^{2n})$, 
then 
$e^{i (\phi_1 (\xi)+\phi_2 (\eta) )} \sigma (\xi, \eta)
\in \calM (H^1 \times L^{\infty}\to L^1)$.  
\end{thm}
%%%=============================

Theorem \ref{th-E1D4} can be deduced from this theorem and from 
Theorem \ref{th-C2}. 
In fact, notice that, by obvious symmetry, 
we have 
$e^{i (\phi_1 (\xi)+\phi_2 (\eta) )} \sigma (\xi, \eta) 
\in 
\calM (L^{\infty} \times H^1 \to L^1)$ 
under the same assumptions on $n$ and $\sigma$. 
Hence, 
if Theorem \ref{th-D4} is proved, 
then we can deduce the claims of 
Theorem \ref{th-E1D4} 
from the claims of Theorems \ref{th-C2} and \ref{th-D4} 
with the aid of complex interpolation. 
(For the interpolation argument, see, e.g., 
\cite[Proof of Theorem 2.2]{BBMNT} 
or \cite[Proof of the `if' part of Theorem 1.1]{Miyachi-Tomita-IUMJ}.)  
Thus it is sufficient to prove Theorem \ref{th-D4}. 

To prove Theorem \ref{th-D4}, we use the following lemmas.

%%%=============================
\begin{lem}\label{lem-D2} 
If $\phi_1, \phi_2 \in \calP (\R^n)$ 
and $\theta \in C_{0}^{\infty} (\R^{2n})$, 
then 
$\left( e^{i (\phi_1 (\xi)+\phi_2 (\eta) )} 
\theta (\xi, \eta)\right)^{\vee} \in L^1 (\R^{2n})$. 
\end{lem}
%%%=============================

\begin{proof}
Take a function $\widetilde{\theta} \in C_{0}^{\infty}(\R^n)$ 
such that $\widetilde{\theta} (\xi) \widetilde{\theta} (\eta)=1$ 
on $\supp \theta$. 
Then 
\[
e^{i (\phi_1 (\xi)+\phi_2 (\eta) )} \theta (\xi, \eta)
=
e^{i (\phi_1 (\xi)+\phi_2 (\eta) )} \widetilde{\theta} (\xi) 
\widetilde{\theta} (\eta) 
\theta (\xi, \eta). 
\]
Lemma \ref{lem-D1} implies 
$\big( e^{i (\phi_1 (\xi)+\phi_2 (\eta) )} \widetilde{\theta} (\xi) 
\widetilde{\theta} (\eta)
\big)^{\vee} \in L^1 (\R^{2n})$. 
Obviously $\big( \theta (\xi, \eta)\big)^{\vee} \in L^1 (\R^{2n})$. 
Hence the conclusion of the lemma follows. 
\end{proof}

%%====================================
\begin{lem}\label{lem-220925}
Let $n\ge 2$ and $\phi \in \calP (\R^n)$ and set 
$R = \sup \big\{ |\nabla \phi (\xi)| \mid |\xi|=1 \big\}$.  
Let $\psi$ be a $C^{\infty}$ function on $\R^n$ 
satisfying $\supp \psi \subset \{2^{-1}\le |\xi|\le 2\}$. 
Then the following hold. 

$(1)$ For each positive integer $N$, 
there  exists a constant $c_N$ depending only on $n, \phi$, and $N$ 
such that 
\[
\left| \big( e^{-i \phi (\xi)} 
\psi (2^{-j}\xi)\big)^{\vee} (x) \right| 
\le c_N \|\psi\|_{C^N} (2^j)^{n-\frac{N}{2}} |x|^{-N} 
\quad \text{for}\;\; 
|x|>2R\;\; \text{and}\;\; j\in \N. 
\]

$(2)$ 
There  exists a constant $c$ depending only on $n$ and $\phi$ 
such that 
\begin{equation*}
\left\| 
\left( e^{- i \phi (\xi) } \psi (2^{-j}\xi)\right)^{\vee} (x)
\right\|_{L^1}
\le 
c  
\|\psi\|_{C^{2n-1}} 
( 2^{j} )^{\frac{n-1}{2}} 
\quad \text{for all}\;\; j \in \N. 
\end{equation*}
\end{lem}
%%====================================

\begin{proof} 
%If $n=1$, the claim is easily proved by a direct 
%calculation. 
%We assume $n\ge 2$ 
%in the rest of the argument. 
%
%
% 
We write 
$f_j (x) = 
\left( e^{-i \phi (\xi)} 
\psi (2^{-j}\xi)\right)^{\vee} (x)$.

To estimate $f_j (x)$, 
we follow the idea given by 
Seeger--Sogge--Stein \cite{SSS}. 
Let 
$S^{n-1}= \{\xi \in \R^n \mid |\xi|=1\}$. 
For each $j\in \N$, 
take a sequence of points $\{\xi_{j}^{\nu}\}_{\nu}$ such that 
\begin{align*}
&
\xi_{j}^{\nu} \in S^{n-1}, 
\\
&
\bigcup_{\nu} 
B(\xi_{j}^{\nu}, 2^{-j/2}) \cap S^{n-1}=
S^{n-1}, 
\\
&
\sum_{\nu} \ichi_{B(\xi_{j}^{\nu}, 2^{-j/2+1})}(\xi) \le c
\;\; \text{for all}\;\;
\xi \in S^{n-1}, 
\end{align*}
where $B(x,r)$ denotes the ball with center $x$ and radius $r$, 
and $\nu$ runs on an index set of cardinality 
$\approx (2^{j/2})^{n-1}$. 
Take 
functions $\{\chi_{j}^{\nu}\}_{\nu}$ such that 
 \begin{align*}
&
\chi_{j}^{\nu}\;\; \text{is homogeneous 
of degree $0$ and $C^{\infty}$ on $\R^n\setminus \{0\}$}, 
\\
&
\{ \xi \in S^{n-1} \mid 
\chi_{j}^{\nu} (\xi) \neq 0 \} 
\subset 
\{\xi \in S^{n-1} \mid |\xi - \xi_{j}^{\nu}|< 2^{-j/2 + 1}\}, 
\\
&
|\partial_{\xi}^{\alpha} 
\chi_{j}^{\nu} (\xi)|
\le c_{\alpha} (2^{j/2})^{|\alpha|}
\;\; \text{for all}\;\;
\xi \in S^{n-1}, 
\\
&
\sum_{\nu} \chi_{j}^{\nu} (\xi)=1 
\;\; \text{for all}\;\;
\xi \in \R^{n}\setminus \{0\}. 
\end{align*}
Using this partition of unity, 
we decompose  
$f_{j}(x)$ as 
\begin{align*}
&
f_{j}(x)
=
\sum_{\nu} 
f_{j}^{\nu}(x), 
\\
&
f_{j}^{\nu}(x)=
\left( 
e^{-i \phi (\xi) }
\psi ( 2^{-j}\xi ) 
\chi_{j}^{\nu}(\xi) 
\right)^{\vee} (x) 
=
\frac{1}{(2\pi)^n} 
\int 
e^{i ( \xi \cdot x -\phi (\xi)) }
\psi ( 2^{-j}\xi ) \chi_{j}^{\nu}(\xi) 
\, d\xi. 
\end{align*}

The key idea is that 
the oscillating factor 
$e^{-i \phi (\xi)}$ can be 
well approximated by $e^{-i \xi \cdot \nabla \phi (\xi_{j}^{\nu}) }$ 
on the support of $\psi (2^{-j}\xi) \chi_{j}^{\nu}(\xi)$. 
We write the phase function $\xi \cdot x - \phi (\xi)$ 
appearing in the last integral 
as 
\begin{align*}
&\xi \cdot x -\phi (\xi)
=
\xi \cdot \big(x - \nabla \phi (\xi)\big)
=
\xi \cdot \big(x - \nabla \phi (\xi_{j}^{\nu}) \big)
+ 
h_{j}^{\nu}(\xi), 
\\
&
h_{j}^{\nu}(\xi)
=
\xi \cdot \big(
\nabla \phi (\xi_{j}^{\nu})
- \nabla \phi (\xi)
\big).  
\end{align*}
Then 
\[
f_{j}^{\nu}(x)=
\frac{1}{(2\pi)^{n}}
\int_{\R^n} 
e^{i \xi\cdot (x- \nabla \phi (\xi_{j}^{\nu}))} \, 
\psi ( 2^{-j}\xi ) 
\chi_{j}^{\nu}(\xi) 
e^{i h_{j}^{\nu}(\xi)}\, d\xi. 
\]
Notice that the support of $\psi ( 2^{-j}\xi ) 
\chi_{j}^{\nu}(\xi) $ is included in the set 
\[
E_{j, \nu} = \left\{ \xi \; \bigg|\;   
2^{j-1} \le |\xi| \le 2^{j+1}, 
\;\; \bigg|\frac{\xi}{|\xi|} - \xi_{j}^{\nu}\bigg|  \le 
2^{-{j}/{2} +1}\right\}, 
\]
which has Lebesgue measure 
$|E_{j, \nu} |\approx (2^j)^{\frac{n+1}{2}}$. 
The functions appearing in the above integral 
satisfy the following estimates 
on $E_{j, \nu} $:  
\begin{align*}
&
|\partial_{\xi}^{\alpha} 
\psi (2^{-j}\xi)|
\le c_{\alpha} \|\psi \|_{C^{|\alpha|}} 
(2^{j})^{-|\alpha|}, 
\\
&
|\partial_{\xi}^{\alpha} 
\chi_{j}^{\nu} (\xi)|
\le c_{\alpha} (2^{\frac{j}{2}})^{-|\alpha|},  
\\
&
|(\xi_{j}^{\nu} \cdot \nabla_{\xi})^{k} 
\chi_{j}^{\nu} (\xi)|
\le c_{k} (2^{j})^{-k}, 
\\
&
|\partial_{\xi}^{\alpha} 
e^{i h_{j}^{\nu} (\xi)}|
\le c_{\alpha} (2^{\frac{j}{2}})^{-|\alpha|}, 
\\
&
|(\xi_{j}^{\nu} \cdot \nabla_{\xi})^{k} 
e^{i h_{j}^{\nu} (\xi)}|
\le c_{k} (2^{j})^{-k}. 
\end{align*}
By using these estimates and 
by integration by parts, 
we obtain the following two estimates:  
\begin{align}
&
|f_{j}^{\nu}(x)|
\le 
c_{N} 
\|\psi \|_{C^N} 
(2^{j})^{\frac{n+1}{2}}
\big(1+2^{\frac{j}{2}} |x- \nabla \phi (\xi_{j}^{\nu}) |\big)^{-N}, 
\label{fjnu-2j/2}
\\
&
|f_{j}^{\nu}(x)|
\le 
c_{N} 
\|\psi \|_{C^N} 
(2^{j})^{\frac{n+1}{2}}
\big(1+2^j |\xi_{j}^{\nu} \cdot (x- \nabla \phi (\xi_{j}^{\nu}) )|\big)^{-N}. 
\label{fjnu-2j}
\end{align}

{\it Proof of (1).}\/  
Suppose $|x|>2 R$. 
Then $|x- \nabla \phi (\xi_{j}^{\nu})|\approx |x|$ 
and hence 
\eqref{fjnu-2j/2} gives 
\[
|f_{j}^{\nu}(x)|
\lesssim 
\|\psi \|_{C^N} 
(2^{j})^{\frac{n+1}{2}}
\big( 2^{\frac{j}{2}} |x|\big)^{-N}. 
\] 
Taking sum over $\nu$'s of card $\approx (2^{\frac{j}{2}})^{n-1}$, 
we have 
\[
|f_j (x)|\le \sum_{\nu} 
|f_{j}^{\nu}(x)|
\lesssim 
\|\psi \|_{C^N} 
(2^{j})^{\frac{n+1}{2}}
\big( 2^{\frac{j}{2}} |x|\big)^{-N} (2^{\frac{j}{2}})^{n-1}
=
\|\psi \|_{C^N} 
\big( 2^{j} \big)^{n -\frac{N}{2}} 
|x|^{-N}. 
\]

{\it Proof of (2).}\/  
Combining   \eqref{fjnu-2j} and \eqref{fjnu-2j/2}, we have 
\begin{align*}
|f_{j}^{\nu}(x)|
&
\le 
c_N 
\|\psi \|_{C^N} 
(2^{j})^{\frac{n+1}{2}}
\big(1+2^j |\xi_{j}^{\nu} \cdot (x- \nabla \phi (\xi_{j}^{\nu}) )|\big)^{-\frac{N}{2}}
\big(1+2^{\frac{j}{2}} |x- \nabla \phi (\xi_{j}^{\nu})| \big)^{-\frac{N}{2}}
\\
&
\le 
c_N \|\psi \|_{C^N} 
(2^{j})^{\frac{n+1}{2}}
\big(1+2^j 
\big|\xi_{j}^{\nu} \cdot (x- \nabla \phi (\xi_{j}^{\nu}) ) \big|
\big)^{-\frac{N}{2}}
\big(1+2^{\frac{j}{2}} 
\big|(x- \nabla \phi (\xi_{j}^{\nu}) )^{\prime}\big| 
\big)^{-\frac{N}{2}}, 
\end{align*}
where 
$(x- \nabla \phi (\xi_{j}^{\nu}) )^{\prime}$ denotes the orthogonal 
projection of $x- \nabla \phi (\xi_{j}^{\nu})$ to the orthogonal complement of 
the line $\R \xi_{j}^{\nu} $. 
Taking $N=2n-1$ and integrating the above inequality 
we have 
\begin{align*}
\|f_{j}^{\nu}\|_{L^1}
&
\lesssim  
\|\psi \|_{C^{2n-1}} 
(2^{j})^{\frac{n+1}{2}}
\int_{\R^n}
\big(1+2^j 
\big|\xi_{j}^{\nu} \cdot (x- \nabla \phi (\xi_{j}^{\nu}) ) \big|
\big)^{-\frac{2n-1}{2}}
\big(
1+2^{\frac{j}{2}} 
\big| \big(x- \nabla \phi (\xi_{j}^{\nu}) \big)^{\prime} \big| 
\big)^{-\frac{2n-1}{2}}\, dx
\\
&
\approx  
\|\psi \|_{C^{2n-1}}.  
\end{align*}
Taking sum over $\nu$'s of card $\approx (2^{\frac{j}{2}})^{n-1}$, 
we obtain the inequality as mentioned in (2). 
This completes the proof of 
Lemma \ref{lem-220925}. 
\end{proof}

%%%=============================
\begin{lem}\label{lem-D3} 
Let $n \ge 2$, 
$\phi \in \calP (\R^n)$, and set 
$R = \sup \big\{ |\nabla \phi (\xi)| \mid |\xi|=1 \big\}$. 
Let $\zeta$ be the function given in Notation \ref{notation} 
and let 
$\theta \in C_{0}^{\infty} (\R^{n})$ 
satisfy $\supp \theta \subset \{|\xi|\le 2\}$. 
Then the following hold. 

$(1)$ 
For each positive integer $N> 2n$, 
there exists a constant $c_N$ depending only on $n, \phi$, and $N$ 
such that
\begin{equation*}
\left| 
\left( 
e^{i \phi (\xi)} 
\zeta (\xi) \theta (2^{-j}\xi) 
\right)^{\vee} (x)
\right| 
\le c_N \|\theta\|_{C^N} \, |x|^{-N}
\quad 
\text{for}\;\; 
|x|>2R\;\; 
\text{and for all}\;\; j \in \N.   
\end{equation*}

$(2)$ 
There exists a constant $c$ 
depending only on $n$ and $\phi$ 
such that
\begin{equation*}
\left\| 
\left( 
e^{i \phi (\xi)} 
\zeta (\xi) \theta (2^{-j}\xi) 
\right)^{\vee} 
\right\|_{L^1}
\le c \|\theta\|_{C^{2n-1}} \, 
{2^{j \frac{n-1}{2}}} . 
\end{equation*}
\end{lem}
%%%=============================

\begin{proof} 
From the definition of $\zeta$ and 
from the assumption on $\supp \theta$, we have 
\[
e^{i \phi (\xi)} 
\zeta (\xi) \theta (2^{-j}\xi) 
=
\sum_{k=1}^{j+1} 
e^{i \phi (\xi)} 
\psi (2^{-k}\xi) \theta (2^{-j}\xi) . 
%=
%\sum_{k=1}^{j+1} 
%e^{i \phi_1 (\xi)} 
%\psi (2^{-k}\xi) \phi (2^{k-j}\cdot 2^{-k}\xi) 
\]

If $|x|>2R$ and $1\le k\le j+1$, then 
Lemma \ref{lem-220925} (1) gives 
\begin{align*}
\left| \left( e^{i \phi (\xi)} 
\psi (2^{-k}\xi) \theta  (2^{-j}\xi)
\right)^{\vee}(x) \right|
&
\lesssim 
(2^{k}) ^{n-\frac{N}{2}}
|x|^{-N}
\left\| 
\psi (\cdot ) \theta (2^{k-j}\cdot )
\right\|_{C^{N}} 
\\
&
\lesssim 
(2^{k}) ^{n-\frac{N}{2}}
|x|^{-N}
\left\| 
\theta 
\right\|_{C^{N}} . 
\end{align*}
If $N>2n$, then taking sum over $k$, 
we obtain the inequality mentioned in (1). 

For $1\le k\le j+1$, 
Lemma \ref{lem-220925} (2) gives 
\begin{equation*}
\left\|
\left( e^{i \phi (\xi)} 
\psi (2^{-k} \xi) \theta (2^{-j}\xi)
\right)^{\vee}(x) 
\right\|_{L^1}
\lesssim 
\left\| 
\psi (\cdot ) \theta (2^{k-j}\cdot )
\right\|_{C^{2n-1}}\,  
(2^{k})^{\frac{n-1}{2}} 
\lesssim 
\left\| 
\theta 
\right\|_{C^{2n-1}}\,  
(2^{k})^{\frac{n-1}{2}} . 
\end{equation*}
Taking sum over $k\le j+1$,  
we obtain the inequality mentioned in (2). 
Lemma \ref{lem-D3} is proved.  
\end{proof}

\begin{proof}[Proof of Theorem \ref{th-D4}] 
We write $m=-n/2$ and assume $\sigma \in S^{m}_{1,0}(\R^{2n})$.

We use the dyadic partition of unity to decompose $\sigma (\xi, \eta)$. 
Let $\psi$, $\varphi$, and $\zeta$ be the functions as given in 
Notation \ref{notation}.  
For $j \in \N \cup \{0\}$, we define $\psi_j$ by 
\[
\psi_j (\xi)= 
\begin{cases}
{\varphi (\xi)} & \text{if}\;\; j=0,  \\
{\psi (2^{-j} \xi)} & \text{if}\;\; j\ge 1. 
\end{cases}
\]
Notice that 
$\sum_{j=0}^{\infty} \psi_j (\xi) =1$ 
and 
$\sum_{j=0}^{k} \psi_j (\xi) = \varphi (2^{-k} \xi)$ 
for $k\in \N \cup \{0\}$.

We decompose $\sigma$ as 
\begin{align*}
\sigma (\xi, \eta) 
&=
\sum_{j=0}^{\infty} 
\sum_{k=0}^{\infty} 
\sigma (\xi, \eta) \psi_j (\xi) \psi_k (\eta)
\\
&
=\sum_{j>k}
+ 
\sum_{j=k}
+
\sum_{j<k}
\\
&
=
\sigma_{\RomI} (\xi, \eta)
+
\sigma_{\II} (\xi, \eta)
+
\sigma_{\III} (\xi, \eta), 
\end{align*}
where 
$\sum_{j>k}$, 
$\sum_{j=k}$, and 
$\sum_{j<k}$ 
denote the sums of 
$\sigma (\xi, \eta) \psi_j (\xi) \psi_k (\eta)$ 
over $(j,k)\in (\N \cup \{0\})^2$ that 
satisfy the designated restrictions.

Consider the multiplier $\sigma_{\RomI}$. 
This is written as 
\[
\sigma_{\RomI}(\xi, \eta)
=
\sum_{j=1}^{\infty} \sum_{k=0}^{j-1} 
\sigma (\xi, \eta) \psi_j (\xi) \psi_k (\eta)
=
\sum_{j=1}^{\infty} 
\sigma (\xi, \eta) \psi (2^{-j}\xi) \varphi (2^{-j+1}\eta). 
\]
Take a function $\widetilde{\psi}\in C_{0}^{\infty}(\R^n)$ such 
that 
$\supp \widetilde{\psi} \subset \{3^{-1}\le |\xi| \le 3\}$ 
and $\widetilde{\psi} (\xi)=1$  
for $2^{-1}\le |\xi| \le 2$. 
Also 
take 
a function $\widetilde{\varphi}\in C_{0}^{\infty}(\R^n)$ such 
that 
$\supp \widetilde{\varphi} \subset \{|\xi| \le 3\}$ 
and $\widetilde{\varphi} (\xi)=1$  
for $|\xi| \le 2$. 
Then 
\[
\sigma_{\RomI}(\xi, \eta)
=
\sum_{j=1}^{\infty} 
\sigma (\xi, \eta) 
\widetilde{\psi}(2^{-j}\xi) 
\widetilde{\varphi}(2^{-j+1}\eta) 
\psi (2^{-j}\xi) \varphi (2^{-j+1}\eta). 
\]
The function 
$\sigma (2^j \xi, 2^{j-1}\eta)\widetilde{\psi}(\xi) 
\widetilde{\varphi}(\eta)$ 
is supported in $\{3^{-1}\le |\xi|\le 3\}\times 
\{|\eta|\le 3\}$ and satisfies the estimate 
\begin{equation*}
\left|
\partial^{\alpha}_{\xi} 
\partial^{\beta}_{\eta} 
\big\{ 
\sigma (2^j \xi, 2^{j-1}\eta)\widetilde{\psi}(\xi) 
\widetilde{\varphi}(\eta)
\big\} 
\right| 
\le 
C_{\alpha, \beta}\,  2^{jm}
\end{equation*}
with $C_{\alpha, \beta}$ independent of $j\in \N$. 
Hence by the Fourier series expansion we can write 
\[
\sigma (2^j \xi, 2^{j-1}\eta)\widetilde{\psi}(\xi) 
\widetilde{\varphi}(\eta)
=
\sum_{a,b \in \Z^n} 
c_{\RomI, j}^{(a,b)} 
e^{i a \cdot \xi} e^{i b \cdot \eta}, 
\quad 
|\xi|<\pi, \;\; |\eta|<\pi,  
\]
with the coefficient satisfying 
\begin{equation}\label{cjab-decay}
\left| c_{\RomI, j}^{(a,b)} \right| 
\lesssim 
2^{jm} 
(1+|a|)^{-L} (1+|b|)^{-L}
\end{equation}
for any $L>0$. 
Changing variables 
$\xi \to 2^{-j}\xi$ and $\eta \to 2^{-j+1}\eta$
and 
multiplying $\psi (2^{-j}\xi) \varphi (2^{-j+1}\eta)$, 
we obtain 
\[
\sigma (\xi, \eta)
\psi (2^{-j}\xi) 
\varphi (2^{-j+1}\eta)
=
\sum_{a,b \in \Z^n} 
c_{\RomI, j}^{(a,b)} 
e^{i a \cdot 2^{-j} \xi} e^{i b \cdot 2^{-j+1}\eta}
\psi (2^{-j}\xi) \varphi (2^{-j+1}\eta).  
\]
Hence $\sigma_{\RomI}$ is written as follows: 
\begin{align*}
\sigma_{\RomI} (\xi, \eta)
&=
\sum_{a,b \in \Z^n}
\sum_{j=1}^{\infty}
c_{\RomI, j}^{(a,b)} 
e^{i a \cdot 2^{-j} \xi} e^{i b \cdot 2^{-j+1}\eta}
\psi (2^{-j} \xi) 
\varphi (2^{-j+1} \eta) 
\\
&
=
\sum_{a,b \in \Z^n}
\sum_{j=1}^{\infty}
c_{\RomI, j}^{(a,b)} 
\psi^{(a)} (2^{-j} \xi) 
\varphi^{(b)} (2^{-j+1} \eta), 
\end{align*}
where 
\begin{equation}\label{psinu-varphinu}
\psi^{(\nu)} (\xi)=e^{i \nu \cdot \xi}\psi (\xi), 
\quad  
\varphi^{(\nu)} (\eta)=e^{i \nu \cdot \eta}\varphi (\eta), 
\quad 
\nu \in \Z^n. 
\end{equation}

By a similar argument, 
$\sigma_{\II}$ and $\sigma_{\III}$ can be written as follows: 
\begin{align*}
&
\sigma_{\II} (\xi, \eta)
=
\sigma (\xi, \eta) \psi_0 (\xi) \psi_0 (\eta)
+ 
\sum_{a,b \in \Z^n}
\sum_{j=1}^{\infty}
c_{\II, j}^{(a,b)} 
\psi^{(a)} (2^{-j} \xi) 
\psi^{(b)} (2^{-j} \eta) , 
\\
&\sigma_{\III} (\xi, \eta)
=
\sum_{a,b \in \Z^n}
\sum_{j=1}^{\infty}
c_{\III, j}^{(a,b)} 
\varphi^{(a)} (2^{-j+1} \xi) 
\psi^{(b)} (2^{-j} \eta), 
\end{align*}
where the coefficients 
$c_{\II, j}^{(a,b)}$ and $c_{\III, j}^{(a,b)}$ 
satisfy the same estimates as \eqref{cjab-decay}.

Hereafter we shall consider a slightly general multiplier. 
We assume the multiplier $\widetilde{\sigma}$ is given by 
\begin{equation}\label{assumption-sigma}
\widetilde{\sigma} (\xi, \eta) 
=
\sum_{j=1}^{\infty}
c_{j} 
\theta_1 (2^{-j} \xi) 
\theta_2 (2^{-j} \eta), 
\end{equation}
where 
$(c_j)_{j\in \N}$ is a sequence of complex numbers satisfying 
\begin{equation}\label{assumption-cj}
\left| c_{j} \right| 
\le 2^{jm} A, \quad j\in \N, 
\end{equation}
with some $A\in (0, \infty)$,  
and $\theta_1$ and $\theta_2$ are 
functions in $C_0^{\infty} (\R^n)$ such that 
\begin{equation}\label{assumption-theta}
\supp \theta_1 , \, \supp \theta_2 
\subset 
\{|\xi|\le 2\}. 
\end{equation}
For such $\widetilde{\sigma}$, we shall prove the estimate 
\begin{equation}\label{Goal-estimate}
\left\| 
e^{i (\phi_1 (\xi)+\phi_2 (\eta))}\, 
\widetilde{\sigma} (\xi, \eta) 
\right\|_{\calM (H^1 \times L^{\infty}\to L^1)}
\le c A \|\theta_1\|_{C^N} \|\theta_2\|_{C^N},  
\end{equation}
with  
$c=c(n)\in (0, \infty)$ and 
$N=N(n)\in \N$.

If this is proved, then by applying it to 
$c_j = c_{\RomI, j}^{(a,b)}$,  
$\theta_1=\psi^{(a)}$, and $\theta_2 = \varphi^{(b)}(2 \cdot)$,  
we obtain 
\begin{align*}
&
\left\| 
e^{i (\phi_1 (\xi)+\phi_2 (\eta))}
\sum_{j=1}^{\infty}
c_{\RomI, j}^{(a,b)} 
\psi^{(a)} (2^{-j} \xi) 
\varphi^{(b)} (2^{-j+1} \eta)
\right\|_{\calM (H^1 \times L^{\infty}\to L^1)}
\\
&
\lesssim (1+|a|)^{-L}
(1+|b|)^{-L} \|\psi^{(a)}\|_{C^N} \|\varphi^{(b)}(2 \cdot)\|_{C^N} 
\lesssim (1+|a|)^{-L+N}
(1+|b|)^{-L+N}, 
\end{align*}
and, thus, taking $L$ sufficiently large and taking sum over 
$a, b \in \Z^n$, 
we obtain 
\[
e^{i (\phi_1 (\xi)+\phi_2 (\eta))} \sigma_{\RomI}(\xi, \eta)
\in \calM (H^1 \times L^{\infty}\to L^1). 
\]  
In the same way, we obtain 
\[
e^{i (\phi_1 (\xi)+\phi_2 (\eta))} \big( \sigma_{\II}(\xi, \eta) - 
\sigma (\xi, \eta) \psi_0 (\xi)\psi_0(\eta)\big)
\in \calM (H^1 \times L^{\infty}\to L^1)
\]
and 
\[
e^{i (\phi_1 (\xi)+\phi_2 (\eta))} \sigma_{\III}(\xi, \eta)
\in \calM (H^1 \times L^{\infty}\to L^1). 
\]
Since $e^{i (\phi_1 (\xi)+\phi_2 (\eta))} 
\sigma (\xi, \eta)\psi_0 (\xi)\psi_0(\eta)$ 
is also a multiplier for 
$H^1 \times L^{\infty}\to L^1$ by virtue of Lemma \ref{lem-D2}, 
we will obtain the conclusion of the theorem.

Thus the proof is reduced to showing \eqref{Goal-estimate} 
for $\widetilde{\sigma}$ given by \eqref{assumption-sigma}, 
\eqref{assumption-cj}, and \eqref{assumption-theta}.

We shall make a further reduction. 
As in the proof of Theorem \ref{th-C2}, 
using the functions $\varphi$ and $\zeta$ 
of Notation \ref{notation}, 
we decompose the multiplier 
$e^{i (\phi_1 (\xi)+\phi_2 (\eta))}\, 
\widetilde{\sigma} (\xi, \eta)$ into four parts:  
\begin{align*}
&
e^{i (\phi_1 (\xi)+\phi_2 (\eta))}\, 
\widetilde{\sigma} (\xi, \eta)
=
\tau_1 (\xi, \eta) + \tau_2 (\xi, \eta) + \tau_3 (\xi, \eta) +\tau_4 (\xi, \eta), 
\\
&
\tau_1 (\xi, \eta) = 
e^{i \phi_1 (\xi)} \varphi (\xi) e^{i \phi_2 (\eta)} \varphi (\eta) 
\widetilde{\sigma} (\xi, \eta), 
\\
&
\tau_2 (\xi, \eta) = 
e^{i \phi_1 (\xi)} \zeta (\xi) e^{i \phi_2 (\eta)} \varphi (\eta) 
\widetilde{\sigma} (\xi, \eta), 
\\
&
\tau_3 (\xi, \eta) = 
e^{i \phi_1 (\xi)} \varphi (\xi) e^{i \phi_2 (\eta)} \zeta (\eta) 
\widetilde{\sigma} (\xi, \eta), 
\\
&
\tau_4 (\xi, \eta) = 
e^{i \phi_1 (\xi)} \zeta (\xi) e^{i \phi_2 (\eta)} \zeta (\eta) 
\widetilde{\sigma} (\xi, \eta).  
\end{align*}

The multipliers $\tau_1$, $\tau_2$, and $\tau_3$ are easy to handle. 
For $\tau_1$, its inverse Fourier transform is given by 
\begin{equation*}
\left( \tau_1(\xi, \eta) \right)^{\vee} (x,y)
= 
\sum_{j=1}^{\infty}
c_j 
\left( e^{i \phi_1 (\xi)} \varphi (\xi) \theta_1 (2^{-j} \xi) \right)^{\vee} (x) 
\left(  e^{i \phi_2 (\eta)} \varphi (\eta) \theta_2 (2^{-j} \eta) \right)^{\vee}(y).  
\end{equation*}
By Lemma \ref{lem-D1}, we have 
\begin{equation}\label{L1-varphi-theta1}
\left\| 
\big( e^{i \phi_1 (\xi)} \varphi (\xi) \theta_1 (2^{-j} \xi) \big)^{\vee} 
\right\|_{L^1} 
\lesssim 
\|\theta_1\|_{C^{n+1}}
\end{equation}
and similar estimate 
with $\theta_2$ in place of $\theta_1$. 
Thus 
\begin{align*}
\left\| \left( \tau_1 \right)^{\vee} \right\|_{L^1 (\R^{2n})}
&
\le 
\sum_{j=1}^{\infty} 
2^{jm} A 
\left\| 
\left( e^{i \phi_1 (\xi)} \varphi (2^{-j} \xi) \theta_1 (2^{-j} \xi) \right)^{\vee} 
\right\|_{L^1(\R^n)} 
\left\| 
\left( e^{i \phi_2 (\eta)} \varphi (\eta) \theta_2 (2^{-j} \eta) \right)^{\vee} 
\right\|_{L^1(\R^n)} 
\\
&
\lesssim 
\sum_{j=1}^{\infty} 
2^{jm} 
A 
\|\theta_1\|_{C^{n+1}}
\|\theta_2\|_{C^{n+1}} 
\approx 
A 
\|\theta_1\|_{C^{n+1}}
\|\theta_2\|_{C^{n+1}}, 
\end{align*}
which implies
\begin{equation*}
\left\| \tau_1 \right\|_{\calM (H^1 \times L^{\infty} \to L^1)} 
\lesssim 
A 
\|\theta_1\|_{C^{n+1}}
\|\theta_2\|_{C^{n+1}}. 
\end{equation*}
For $\tau_2$, 
we use the estimate 
\begin{equation}\label{L1-zeta-theta1}
\left\| 
\left( e^{i \phi_1 (\xi)} \zeta (\xi) \theta_1 (2^{-j} \xi) \right)^{\vee} 
\right\|_{L^1(\R^n)} 
\lesssim 2^{j\frac{n-1}{2}}\|\theta_1\|_{C^{2n-1}}, 
\end{equation}
which is given in Lemma \ref{lem-D3} (2). 
Using this together with \eqref{L1-varphi-theta1}, we obtain 
\begin{align*}
&
\left\| \tau_2 \right\|_{\calM (H^1 \times L^{\infty} \to L^1)} 
\le 
\left\| \left( \tau_2 \right)^{\vee} \right\|_{L^1 (\R^{2n})}
\\
&
= 
\left\| 
\sum_{j=1}^{\infty}
c_j 
\left( e^{i \phi_1 (\xi)} \zeta (\xi) 
\theta_1 (2^{-j} \xi) \right)^{\vee} (x) 
\left(  e^{i \phi_2 (\eta)} \varphi (\eta) 
\theta_2 (2^{-j} \eta) \right)^{\vee}(y)
\right\|_{L^1_{x,y}(\R^{2n})} 
\\
&
\lesssim 
\sum_{j=1}^{\infty} 
2^{jm} 
A \, 
2^{j\frac{n-1}{2}} 
\|\theta_1\|_{C^{2n-1}}
\|\theta_2\|_{C^{n+1}} 
\approx 
A 
\|\theta_1\|_{C^{2n-1}}
\|\theta_2\|_{C^{n+1}},  
\end{align*}
where the last $\approx$ holds because 
$m<-(n-1)/2$. 
Similarly, we have 
\[
\left\| \tau_3 \right\|_{\calM (H^1 \times L^{\infty} \to L^1)} 
\le 
\left\| \left( \tau_3 \right)^{\vee} \right\|_{L^1 (\R^{2n})}
\lesssim 
A 
\|\theta_1\|_{C^{n+1}}
\|\theta_2\|_{C^{2n-1}}. 
\]

Thus the rest of the proof is the estimate for $\tau_4$. 
Our purpose is to prove the estimate 
\[
\left\|
T_{\tau_4} (f,g) 
\right\|_{L^1}
\lesssim 
A
\|\theta_1\|_{C^{N}}
\|\theta_2\|_{C^{N}}
\|f\|_{H^1} 
 \|g\|_{L^{\infty}}. 
\]
To prove this, by virtue of the atomic decomposition of $H^1$, 
it is sufficient to prove the uniform estimate of 
$\left\|
T_{\tau_4} (f,g) 
\right\|_{L^1}$ for $H^1$-atoms $f$. 
By translation, we may assume that the 
$H^1$-atoms are supported on balls centered at the origin. 
Thus we assume 
\[
\supp f \subset \{|x|\le r\}, 
\quad 
\|f\|_{L^{\infty}}\le r^{-n}, 
\quad 
\int f(x)\, dx = 0, 
\]
and we shall prove 
\[
\left\|
T_{\tau_4} (f,g) 
\right\|_{L^1}
\lesssim 
A
\|\theta_1\|_{C^{N}}
\|\theta_2\|_{C^{N}}
 \|g\|_{L^{\infty}}. 
\]
Recall that the bilinear operator $T_{\tau_4}$ is given by 
\[
T_{\tau_4} (f,g)(x)
=
\sum_{j=1}^{\infty} c_j 
\left( 
e^{i \phi_1 (D) } \zeta (D) \theta_1 (2^{-j} D) f\right) (x) 
\left( 
e^{i \phi_2 (D)} \zeta (D) \theta_2 (2^{-j} D) g\right) (x). 
\]

We set $R = 1+ \max_{i=1,2} \sup \big\{ |\nabla \phi_i (\xi)| \mid |\xi|=1 \big\}$.

Firstly, consider the case $r>R$. 
In this case we estimate the $L^1$ norm as 
\begin{align*}
\left\|
T_{\tau_4} (f,g) 
\right\|_{L^1}
&
\le 
\sum_{j=1}^{\infty} 
2^{jm} A 
\left\| 
e^{i \phi_1 (D)} \zeta (D) \theta_1 (2^{-j} D) f
\right\|_{L^1}
\left\|  
e^{i \phi_2 (D)} \zeta (D) \theta_2 (2^{-j} D) g
\right\|_{L^{\infty}}
\\
&
=: (\ast). 
\end{align*}
For the $L^\infty$-norm involving $g$, 
we use Lemma \ref{lem-D3} (2) to obtain 
\begin{equation}\label{SjgLinfty}
\begin{split}
\left\| 
e^{i \phi_2 (D)} \zeta (D) \theta_2 (2^{-j} D) g 
\right\|_{L^{\infty}}
&
\le 
\left\| 
\left( e^{i \phi_2 (\eta)} \zeta (\eta) \theta_2 (2^{-j}\eta) 
\right)^{\vee} 
\right\|_{L^1} \|g\|_{L^{\infty}}
\\
&
\lesssim 
2^{j \frac{n-1}{2}}\|\theta_2\|_{C^{2n-1}}
\|g\|_{L^{\infty}}. 
\end{split}
\end{equation}
For the $L^1$ norm of 
$e^{i \phi_1 (D)} \zeta (D) \theta_1 (2^{-j} D) f (x)$ 
on $|x|\le 3r$, 
we use the Cauchy--Schwarz inequality to obtain 
\begin{align*}
\left\| 
\left( e^{i \phi_1 (D)} \zeta (D) \theta_1 (2^{-j} D) f\right) (x)
\right\|_{L^1(|x|\le 3r)}
&\lesssim 
r^{n/2} 
\left\| 
\left( e^{i \phi_1 (D)} \zeta (D) \theta_1 (2^{-j} D) f\right) (x)
\right\|_{L^2(|x|\le 3r)}
\\
&
\lesssim 
r^{n/2} \|\theta_1\|_{C^0} \|f\|_{L^2}
\lesssim 
\|\theta_1\|_{C^0}.  
\end{align*}
If $|x|>3r$ and $|y|\le r$, 
then 
$|x-y|>2r >2R$. 
Hence, for $|x|>3r$, using Lemma \ref{lem-D3} (1), we see that 
\begin{align*}
&
\left| 
e^{i \phi_1 (D)} \zeta (D) \theta_1 (2^{-j} D) f (x)
\right|
\\
&
=
\left| 
\int 
\left( e^{i \phi_1 (\xi)} \zeta (\xi) \theta_1 (2^{-j} \xi) 
\right)^{\vee} (x-y) \, f(y)\, dy  
\right|
\\
&
\lesssim 
\int_{|y|\le r} 
\|\theta_1\|_{C^N} 
|x-y|^{-N} \, |f(y)|\, dy  
\lesssim 
\|\theta_1\|_{C^N}  
|x|^{-N}, 
\end{align*}
which implies 
\[
\left\| 
 e^{i \phi_1 (D)} \zeta (D) \theta_1 (2^{-j} D) f (x)
\right\|_{L^1(|x|> 3r)}
\lesssim 
 \|\theta_1\|_{C^N} 
\int_{|x|>3r} |x|^{-N}\, dx   
\lesssim 
 \|\theta_1\|_{C^N} . 
\]
Combining the above estimates, 
we have 
\begin{equation}\label{SjfL1}
\left\| 
e^{i \phi_1 (D)} \zeta (D) \theta_1 (2^{-j} D) f (x)
\right\|_{L^1}
\lesssim 
 \|\theta_1\|_{C^N}.  
 \end{equation}
Now from \eqref{SjgLinfty} and 
\eqref{SjfL1}, 
we obtain 
\[
(\ast) 
\lesssim 
\sum_{j=1}^{\infty} 
2^{jm} A 
\|\theta_1\|_{C^N} 
2^{j \frac{n-1}{2}}\|\theta_2\|_{C^N}
\|g\|_{L^{\infty}}
\approx 
A \|\theta_1\|_{C^N} 
\|\theta_2\|_{C^N}
\|g\|_{L^{\infty}}, 
\]
where the last $\approx$ holds because $m< -(n-1)/2$.

Secondly, we assume $r\le R$ and estimate the $L^1$ norm of 
$T_{\tau_4}(f,g)(x)$ on $|x|>3R$. 
We estimate this as 
\begin{align*}
&\left\|
T_{\tau_4} (f,g) (x) 
\right\|_{L^1(|x|>3R)}
\\
&
\le 
\sum_{j=1}^{\infty} 
2^{jm} A 
\left\| 
e^{i \phi_1 (D)} \zeta (D) \theta_1 (2^{-j} D) f (x) 
\right\|_{L^1(|x|>3R)}
\left\|  
e^{i \phi_2 (D)} \zeta (D) \theta_2 (2^{-j} D) g (x) 
\right\|_{L^{\infty}(|x|>3R)}
\\
&
=: (\ast \ast). 
\end{align*}
For the $L^{\infty}$ norm involving $g$, we have 
\eqref{SjgLinfty}. 
If $|x|>3R$ and $|y|\le r \le R$, then 
$|x-y|>2R$. 
Hence, 
for $|x|>3R$, 
Lemma \ref{lem-D3} (1) yields 
\begin{align*}
&
\left| 
e^{i \phi_1 (D)} \zeta (D) \theta_1 (2^{-j} D) f (x) 
\right|
=
\left| 
\int 
\left( e^{i \phi_1 (\xi)} \zeta (\xi) \theta_1 (2^{-j} \xi) \right)^{\vee} 
(x-y)\,  f (y)\, dy 
\right|
\\
&
\lesssim 
\int_{|y|\le r}   
\| \theta_1 \|_{C^{N}} |x-y|^{-N}\,  |f (y)|\, dy 
\lesssim 
\| \theta_1 \|_{C^{N}} |x|^{-N}. 
\end{align*}
This implies 
\[
\left\| 
e^{i \phi_1 (D)} \zeta (D) \theta_1 (2^{-j} D) f (x) 
\right\|_{L^1(|x|>3R)} 
\lesssim 
\| \theta_1 \|_{C^{N}}. 
\]
Thus we obtain 
\[
(\ast \ast) 
\lesssim 
\sum_{j=1}^{\infty} 
2^{jm} A 
\|\theta_1\|_{C^N} 
2^{j \frac{n-1}{2}}\|\theta_2\|_{C^N}
\|g\|_{L^{\infty}}
\approx 
A \|\theta_1\|_{C^N} 
\|\theta_2\|_{C^N}
\|g\|_{L^{\infty}}, 
\]
where we used $m< -(n-1)/2$ again.

Thirdly, we assume $r\le R$ and estimate the $L^1$ norm of 
$T_{\tau_4}(f,g)(x)$ on $|x|\le 3R$. 
We set 
$B=\{x \in \R^n \mid |x|\le 5R\}$ and decompose 
$g$ as $g=g\ichi_{B}+g\ichi_{B^c}$.

For the $L^1 (|x|\le 3R)$ norm of 
$T_{\tau_4}(f,g\ichi_{B^c})(x)$, we have 
\begin{align*}
&\left\|
T_{\tau_4} (f,g\ichi_{B^c}) (x) 
\right\|_{L^1(|x|\le 3R)}
\\
&
\le 
\sum_{j=1}^{\infty} 
2^{jm} A 
\left\| 
e^{i \phi_1 (D) } \zeta (D) \theta_1 (2^{-j} D) f (x) 
\right\|_{L^1(|x|\le 3R)}
\left\|  
e^{i \phi_2 (D)} \zeta (D) \theta_2 (2^{-j} D) \left( g\ichi_{B^c}\right) (x) 
\right\|_{L^{\infty}(|x|\le 3R)}
\\
&
=: (\ast \ast \ast). 
\end{align*}
Using Lemma \ref{lem-D3} (2), 
we have 
\begin{equation*}
\left\| 
e^{i \phi_1 (D)} \zeta (D) \theta_1 (2^{-j} D) f (x) 
\right\|_{L^1(|x|\le 3R)} 
%&
\le 
\left\| 
\left( e^{i \phi_1 (\xi)} \zeta (\xi) \theta_1 (2^{-j}\xi) 
\right)^{\vee} 
\right\|_{L^1} 
\|f\|_{L^1}
%\\
%&
\lesssim 
2^{j \frac{n-1}{2}} 
\| \theta_1 \|_{C^{N}}. 
\end{equation*}
If $|x|\le 3R$ and $|y|> 5R$, then 
$|x-y|>2R$. 
Hence, 
for $|x|\le 3R$, 
we use Lemma \ref{lem-D3} (1) to have 
\begin{align*}
&
\left| 
e^{i \phi_2 (D)} \zeta (D) \theta_2 (2^{-j} D) 
\left( g\ichi_{B^c} \right) (x) 
\right|
=
\left| 
\int_{|y|>5R} 
\left( e^{i \phi_2 (\eta)} \zeta (\eta) \theta_2 (2^{-j} \eta) \right)^{\vee} 
(x-y)\,  g (y)\, dy 
\right|
\\
&
\lesssim 
\int_{|y|>5R }   
\| \theta_2 \|_{C^{N}} |x-y|^{-N}\,  \| g \|_{L^{\infty}}\, dy 
\approx  
\| \theta_2 \|_{C^{N}} \| g \|_{L^{\infty}}. 
\end{align*}
Thus  
\[
(\ast \ast \ast) 
\lesssim 
\sum_{j=1}^{\infty} 
2^{jm} A \, 2^{j \frac{n-1}{2}}
\|\theta_1\|_{C^N} 
\|\theta_2\|_{C^N}
\|g\|_{L^{\infty}}
\approx 
A \|\theta_1\|_{C^N} 
\|\theta_2\|_{C^N}
\|g\|_{L^{\infty}}, 
\]
where we used $m< -(n-1)/2$ again.

Finally, we estimate the $L^1$ norm of   
$T_{\tau_4} (f, g \ichi_{B})(x)$ on $|x|\le 3R$. 
For this, we use the Cauchy--Schwarz inequality to have 
\begin{align*}
&\left\|
T_{\tau_4} (f,g\ichi_B) (x) 
\right\|_{L^1(|x|\le 3R)}
\\
&
\le 
\sum_{j=1}^{\infty} 
2^{jm} A 
\left\| 
e^{i \phi_1 (D)} \zeta (D) \theta_1 (2^{-j} D) f 
\right\|_{L^2}
\left\|  
e^{i \phi_2 (D)} \zeta (D) \theta_2 (2^{-j} D) \left( g\ichi_B \right)  
\right\|_{L^2}
\\
&
=: (\ast \ast \ast \ast). 
\end{align*}
For the $L^2$ norm involving $g \ichi_B$, 
we have 
\begin{equation}\label{L2g1B}
\left\|  
e^{i \phi_2 (D)} \zeta (D) \theta_2 (2^{-j} D) \left( g\ichi_B \right)  
\right\|_{L^2}
\lesssim  
\left\| \theta_2 \right\|_{C^0}
\left\| g\ichi_B \right\|_{L^2}
\lesssim 
\left\| \theta_2 \right\|_{C^0}
\| g \|_{L^{\infty}}. 
\end{equation}
We estimate 
the $L^2$ norm of $e^{i \phi_1 (D)} \zeta (D) \theta_1 (2^{-j} D) f $ in two ways. 
Firstly, we have 
\begin{equation}\label{L2f-1}
\left\|  
e^{i \phi_1 (D)} \zeta (D) \theta_1 (2^{-j} D) f   
\right\|_{L^2}
\lesssim  
\left\| \theta_1 \right\|_{C^0}
\| f \|_{L^2}
\lesssim 
 r^{-n/2} 
\left\| \theta_1 \right\|_{C^0}. 
\end{equation}
On the other hand, 
using the moment condition of $f$, 
we can write 
\begin{align*}
&
e^{i \phi_1 (D)} \zeta (D) \theta_1 (2^{-j} D) f  (x)
\\
&
= 
\int 
\bigg\{ 
\left( e^{i \phi_1 (\xi)} \zeta (\xi) \theta_1 (2^{-j} \xi) \right)^{\vee} (x-y) 
-
\left( e^{i \phi_1 (\xi)} \zeta (\xi) \theta_1 (2^{-j} \xi) \right)^{\vee} (x) 
\bigg\} 
f(y)\, dy
\\
&
=
\iint_{
\substack{
{0<t<1}
\\
{|y|\le r}  }  }
\nabla 
\left( e^{i \phi_1 (\xi)} \zeta (\xi) \theta_1 (2^{-j} \xi) \right)^{\vee} (x-t y) 
\cdot (-y)\, f(y)\, dt dy. 
\end{align*}
Hence 
\begin{align*}
\left\|  
e^{i \phi_1 (D)} \zeta (D) \theta_1 (2^{-j} D) f   
\right\|_{L^2}
&
\lesssim  
\left\| 
\nabla 
\left( e^{i \phi_1 (\xi)} \zeta (\xi) \theta_1 (2^{-j} \xi) \right)^{\vee} 
\right\|_{L^2} 
\int_{|y|\le r} |y|\, |f(y)|\, dy
\\
&
\lesssim 
r\, 
\left\| 
\nabla 
\left( e^{i \phi_1 (\xi)} \zeta (\xi) \theta_1 (2^{-j} \xi) \right)^{\vee} 
\right\|_{L^2}. 
%\lesssim 
%r\, 
%2^{j (\frac{n}{2}+1)} \|\phi_1\|_{C^0}.  
\end{align*}
By Plancherel's theorem, 
\begin{align*}
\left\| 
\nabla 
\left( e^{i \phi_1 (\xi)} \zeta (\xi) \theta_1 (2^{-j} \xi) \right)^{\vee} 
\right\|_{L^2}
\approx 
\left\| \xi \zeta (\xi) \theta_1 (2^{-j} \xi) \right\|_{L^2}
%\, 2^{j \frac{n}{2}}
%\\
%&
\lesssim 
2^{j (\frac{n}{2}+1)} \|\theta_1\|_{C^0}.  
\end{align*}
Hence 
\begin{equation}\label{L2f-2}
\left\|  
e^{i \phi_1 (D)} \zeta (D) \theta_1 (2^{-j} D) f   
\right\|_{L^2}
\lesssim  
2^{j (\frac{n}{2}+1)}\, r 
\left\| \theta_1 \right\|_{C^0}. 
\end{equation}
Combining \eqref{L2g1B}, \eqref{L2f-1}, and 
\eqref{L2f-2}, we obtain 
\begin{align*}
(\ast \ast \ast \ast) 
&\lesssim 
\sum_{j=1}^{\infty} 
2^{jm} A \, 
\min \left\{ 
r^{-\frac{n}{2}}, \, 
2^{j (\frac{n}{2}+1)}\, r 
\right\} 
\|\theta_1 \|_{C^0}
\|\theta_2\|_{C^0}
\|g\|_{L^{\infty}}
\\
&
=
A \|\theta_1\|_{C^0} 
\|\theta_2\|_{C^0}
\|g\|_{L^{\infty}}
\sum_{j=1}^{\infty} 
\min \left\{ 
\left( 2^j r\right)^{-\frac{n}{2}}, \, 
2^j r 
\right\} 
\\
&
\approx 
A \|\theta_1\|_{C^0} 
\|\theta_2\|_{C^0}
\|g\|_{L^{\infty}}, 
\end{align*}
where we used $m= -n/2$. 
This completes the proof of Theorem \ref{th-D4}. 
\end{proof}

%%%%%%%%%%%%%%%%%%%%%%%%%%%%%%%%%
\section{Necessary conditions on $m$}
\label{E}
%%%%%%%%%%%%%%%%%%%%%%%%%%%%%%%%%

In this section, we shall prove Theorem \ref{th-E1E3}.

In fact, we shall prove a stronger theorem 
by considering a multiplier of a special form.  
Take a function $\theta \in C_{0}^{\infty}(\R^n)$ such that 
$\supp \theta \subset \{3^{-1}\le |\xi|\le 3\}$ and 
$\theta (\xi)=1$ for $2^{-1}\le |\xi|\le 2$. 
We consider the multiplier 
\begin{equation*}
\sigma_j (\xi, \eta) 
=2^{jm} \theta (2^{-j}\xi) \theta (2^{-j}\eta), 
\quad 
j\in \N. 
\end{equation*}
This multiplier satisfies the inequalities 
\begin{equation*}
|\partial^{\alpha}_{\xi} 
\partial^{\beta}_{\eta} \sigma_j (\xi, \eta)| 
\le 
C_{\alpha, \beta} \left( 1+ |\xi| + |\eta|\right)^{m-|\alpha|-|\beta|}
\end{equation*}
with $C_{\alpha, \beta}$ independent of $j \in \N$. 
Thus if the 
the assertion \eqref{eq-nec1} holds 
then, by the closed graph theorem, it follows that 
there exists a constant 
$A=A(n,m,p,q, \theta)$ such that 
\begin{equation}\label{eq-nec2}
\left\| 
2^{jm} 
e^{i (|\xi|+ |\eta|)} 
\theta (2^{-j} \xi) 
\theta (2^{-j} \eta)
\right\| 
_{\calM (H^p \times H^q \to X_r) }
\le A 
\;\; 
\text{for all}
\;\; 
j \in \N.  
\end{equation}
We shall prove that the conditions given in 
Theorem \ref{th-E1E3} are already 
necessary for \eqref{eq-nec2}. 
We shall prove the following theorem, 
which asserts that 
the claims of Theorem 
\ref{th-E1E3} hold if we replace the condition 
\eqref{eq-nec1} by the condition \eqref{eq-nec2}.

%%%=============================
\begin{thm}\label{th-E} 
Let $n \ge 2$. 

$(1)$ Let $0<p, q \le 2$ or $2\le p, q \le \infty$. 
Then 
$m \in \R$ satisfies \eqref{eq-nec2} 
only if 
$m \le - (n-1) \big( |1/p - 1/2| + |1/q - 1/2| \big)$. 

$(2)$ Let $1\le p \le 2\le q \le \infty$ 
or $1\le q \le 2\le p \le \infty$ 
and assume 
$1/p+1/q=1$. 
Then $m \in \R$ satisfies 
\eqref{eq-nec2} 
only if $m \le -n |1/p -1/2|$. 
\end{thm}
%%%=============================

To prove this theorem, we use the following lemma.

%%%=============================
\begin{lem}\label{lem-220919}
Let $\psi$ be a $C^{\infty}$ function on $\R$ 
such that 
\[
\supp \psi \subset \{t\in \R \mid 2^{-1}\le t\le 2\}, 
\quad \psi (t) \ge 0, 
\quad 
\psi (t) \not\equiv 0, 
\] 
and set 
\[
h_j (x) = 
\left( e^{-i |\xi|} \psi (2^{-j} |\xi|)\right)^{\vee} (x), 
\]
which is the inverse Fourier transform of the radial function 
$e^{-i |\xi|} \psi (2^{-j} |\xi|)$ on $\R^n$. 
Then 
the following hold. 
%\begin{itemize}
%\item[$(1)$] 

$(1)$ 
For each $L>0$, 
there exists a constant 
$c_L$, depending only on $n, \psi$, and $L$,	 
such that 
\begin{equation*}
\left|
h_j (x) 
%\left( e^{-i |\xi|} \psi (2^{-j} |\xi|)\right)^{\vee} (x)
\right|
\le 
c_L \, 
2^{j \frac{n+1}{2}}
\left( 1+ 2^{j} \big| 1- |x|\big|\right)^{-L} 
\end{equation*}
for all $ j\in \N$ and all $x\in \R^n$.  
%\item[$(2)$] 

$(2)$ 
There exist 
$\delta, c_0 \in (0, \infty)$ and $j_0\in \N$, 
depending only on 
$n$ and $\psi$, 
such that 
\begin{equation*}
\big|e^{i(\frac{(n-2)\pi}{4}+\frac{\pi}{4})}\, 
2^{-j\frac{n+1}{2}} \, 
h_j (x)
- c_0\, \big|
\le \frac{c_0}{10}
\quad  
\text{if}
\;\; 
1- \delta 2^{-j } < |x|  <1+ \delta 2^{-j}  
\; \; \text{and}\;\; 
j > j_0. 
\end{equation*}

%\item[$(3)$] 

$(3)$ 
For each $0<p\le \infty$, 
\[
\left\|
h_j
\right\|_{H^p}
\approx 
\left\|
h_j
\right\|_{L^p}
\approx 
2^{j (\frac{n+1}{2}-\frac{1}{p})}, 
\quad j\in \N,  
\]
where the implicit constants in $\approx$ depend only on 
$n, \psi$, and $p$. 
%\end{itemize}
\end{lem} 
%%%=============================

\begin{proof} 
The assertion (3) follows from (1) and (2). 
In fact, the inequality $\|h_j\|_{H^p}\approx \|h_j\|_{L^p}$ holds 
because the support of the inverse Fourier transform of $h_j$ is 
included in the annulus $\{2^{j-1}\le |\xi|\le 2^{j+1}\}$. 
The estimate of $\|h_j\|_{L^p}\lesssim 2^{j (\frac{n+1}{2}-\frac{1}{p})}$ 
follows from (1) and 
the converse estimate 
$\|h_j\|_{L^p}\gtrsim 2^{j (\frac{n+1}{2}-\frac{1}{p})}$ 
follows from (2). 
Thus we only need to prove (1) and (2).

Since $h_j (x)$ is the inverse Fourier transform 
of a radial function, 
it is written in terms of 
Bessel function as 
\[
h_j (x) = 
(2\pi)^{-\frac{n}{2}}
\int_{0}^{\infty} 
J_{\frac{n-2}{2}}(|x|t)\, 
(|x|t)^{-\frac{n-2}{2}}\, 
\psi \big( 2^{-j}t \big)\, 
e^{-i t}\, t^{n-1}\, dt, 
\]
where 
$J_{\frac{n-2}{2}}$ is the Bessel function 
(see {\it e.g.}\/ \cite[Theorem 3.3, p.\ 155]{SW}; 
this formula holds for $n=1$ as well since 
$
(2\pi)^{-1/2}
J_{-1/2}(s)\, 
s^{1/2}
={\pi}^{-1} \cos s
$).

{\it Proof of $(1)$.}\/ 
Firstly, we estimate of $h_j (x)$ for $2^{j} |x|\le 1$. 
For this, we use the power series expansion 
\begin{equation*}
(2\pi)^{-\frac{n}{2}
}\, 
J_{\frac{n-2}{2}}(s)\, 
s^{-\frac{n-2}{2}}
=
\sum_{m=0}^{\infty}
a_{m}\, s^{m},  
\end{equation*}
whose radius of convergence is $\infty$. 
Integrating term by term, we have 
\begin{align*}
h_j(x)
%&=
%(2\pi)^{-\frac{n}{2}}
%\int_{0}^{\infty} 
%J_{\frac{n-2}{2}}(|x|t)\, 
%(|x|t)^{-\frac{n-2}{2}}\, 
%\psi \big( 2^{-j} t \big)\, 
%e^{-i t}\, t^{n-1}\, dt
%\\
&=
\int_{0}^{\infty} 
\sum_{m=0}^{\infty}
a_{m}\, (|x|t)^{m}\,
\psi \big( 2^{-j} t \big)\, 
e^{-i t}\, t^{n-1}\, dt
\\
&
=
\sum_{m=0}^{\infty}
a_{m}\, |x|^m
\big(  
\psi \big( 2^{-j} t \big)
\, t^{m+n-1}\big) ^{\wedge}(1)
\\
&
=\sum_{m=0}^{\infty}
a_{m}\, |x|^m
(2^{j})^{m+n} 
\big(  
\psi ( t )
\, t^{m+n-1}\big)^{\wedge}(2^{j}). 
\end{align*}
The function arising in the last expression 
satisfies 
\begin{align*}
&
\supp 
\big( 
\psi ( t )
\, t^{m+n-1}\big)
\subset \{t\in \R \mid 2^{-1}\le t \le 2\}, 
\\
&
\bigg| 
\bigg(\frac{d}{dt}\bigg)^{\ell} 
\big( 
\psi ( t )
\, t^{m+n-1}\big) 
\bigg|
\le c_{\ell }\, 
(1+m)^{\ell}\, 2^{m}. 
\end{align*}
Hence, for any 
$L^{\prime}\in \N$, 
we have 
\[
\bigg|
\big( 
\psi ( t )
\, t^{m+n-1}
\big)^{\wedge}(2^{j})
\bigg| 
\le 
c_{L^{\prime}}
(1+m)^{L^{\prime}}\, 
2^{m}\, 
 \big( 2^{j}\big)^{-L^{\prime}}. 
\]
Thus, 
for $2^{j}|x|\le 1$, 
we have 
\begin{align*}
|h_j(x)|
&
\le 
\sum_{m=0}^{\infty}
|a_{m}||x|^m\,
\big( 2^{j}\big)^{m+n}  
c_{L^{\prime}}  (1+m)^{L^{\prime}}\,2^{m}\, 
 \big( 2^{j}\big)^{-L^{\prime}}
\\
&
\le 
c_{L^{\prime}} 
\big( 2^{j}\big)^{n-L^{\prime}} 
\sum_{m=0}^{\infty}
|a_{m}|\,
(1+m)^{L^{\prime}}\,
2^{m}\, 
= 
\widetilde{c}_{L^{\prime}}  
\big( 2^{j}\big)^{n-L^{\prime}}.  
\end{align*}
Since $L^{\prime}$ can be taken arbitrarily large, 
the above implies the desired estimate of 
$h_j (x)$ for $2^{j}|x|\le 1$.

Next, we estimate $h_j(x)$ for $2^{j} |x|>1$. 
For this,  
we use the asymptotic expansion of the Bessel function, 
which reads as 
\begin{equation*}
(2\pi)^{-\frac{n}{2}}\, 
J_{\frac{n-2}{2}}(s)\, 
s^{-\frac{n-2}{2}}
=
b^{+} e^{i s} 
s^{\frac{1}{2}-\frac{n}{2}}\, 
+ 
b^{-} e^{-i s} 
s^{\frac{1}{2}-\frac{n}{2}}\, 
+
e^{i s}R^{+}(s)
+ 
e^{-i s} R^{-}(s),  
\end{equation*}
where 
$
b^{\pm}= 
(2\pi)^{-\frac{n+1}{2}}
e^{\mp i (\frac{(n-2)\pi}{4}+\frac{\pi}{4})}
$ 
and the remainder terms $R^{\pm}(s)$ satisfy 
\begin{equation}\label{remainder}
\bigg( \frac{d}{ds}\bigg)^{\ell} R^{\pm} (s)
=O\big(s^{\frac{1}{2}-\frac{n}{2}-1-\ell }\big) 
\; \; \text{as}\;\; 
s \to \infty , 
\quad 
\ell = 0, 1, 2, \dots .
\end{equation}
Corresponding to the above formula, 
we write 
\begin{align*}
h_j(x)
& =
b^{+} 
\int_{0}^{\infty} 
e^{i |x|t} 
(|x|t)^{\frac{1}{2}-\frac{n}{2}}\, 
\psi \big( 2^{-j} t \big)\, 
e^{-i t}\, t^{n-1}\, dt
\\
& \quad 
+ 
b^{-} 
\int_{0}^{\infty} 
e^{-i |x|t}  
(|x|t)^{\frac{1}{2}-\frac{n}{2}}\, 
\psi \big( 2^{-j} t \big)\, 
e^{-i t}\, t^{n-1}\, dt
\\
& \quad 
+
\int_{0}^{\infty} 
e^{i |x|t}
R^{+}(|x|t)\, 
\psi \big( 2^{-j} t \big)\, 
e^{-i t}\, t^{n-1}\, dt
\\
& \quad 
+
\int_{0}^{\infty} 
e^{-i |x|t}
R^{-}(|x|t)\, 
\psi \big( 2^{-j} t \big)\, 
e^{-i t}\, t^{n-1}\, dt
\\
& 
=
I^{+}_{j} (x)
+ I^{-}_{j} (x)   
+ K^{+}_{j} (x) 
+K^{-}_{j} (x). 
\end{align*}
We shall estimate 
each of 
$I^{+}_{j} (x)$, 
$I^{-}_{j} (x)$, $K^{+}_{j} (x)$, 
and $K^{-}_{j} (x)$ 
for $2^{j} |x| > 1$.

%%============================
Estimate of $I^{+}_{j} (x)$ for $2^{j} |x|>1$. 
$I^{+}_{j} (x)$ is written as 
\begin{equation}\label{formula-I+j}
\begin{split}
I^{+}_{j} (x)
&=
 b^{+}\, 
\big\{ 
( |x| t )^{\frac{1}{2}-\frac{n}{2}}\, 
\psi (2^{-j}t )\, t^{n-1} 
\big\} ^{\wedge} (1-|x|)
\\
&
=
 b^{+}\,  
|x|^{\frac{1}{2}-\frac{n}{2}} \, 
(2^{j})^{\frac{1}{2}+\frac{n}{2}}\, 
\big( \psi (t)\, t^{-\frac{1}{2}+\frac{n}{2}} \big)^{\wedge}
\big( 2^{j} (1-|x|) \big). 
\end{split}
\end{equation}
Since $\big( \psi (t)\, t^{-\frac{1}{2}+\frac{n}{2}} \big)^{\wedge}$ 
is a rapidly decreasing function, 
we have 
\[
|I^{+}_{j} (x)|
\lesssim 
|x|^{\frac{1}{2}-\frac{n}{2}}\, 
(2^{j})^{\frac{1}{2}+\frac{n}{2}}\, 
\big( 1+ 2^{j} \big|1-|x|\big| \big)^{-L^{\prime}} 
\]
for any $L^{\prime}> 0$. 
Hence, 
\begin{align*}
&
2^{-j}< |x|\le 2^{-1} \; \Rightarrow \;
|I^{+}_{j} (x)|
\lesssim 
(2^{-j})^{^{\frac{1}{2}-\frac{n}{2}}}\, 
(2^{j})^{\frac{1}{2}+\frac{n}{2}}\, 
(2^{j})^{-L^{\prime}}
=
(2^j)^{n-L^{\prime}}, 
\\
&
|x|> 2^{-1}\; \Rightarrow \; 
|I^{+}_{j} (x)|
\lesssim 
(2^{j})^{\frac{1}{2}+\frac{n}{2}}\, 
\big( 1+ 2^{j} \big|1-|x|\big| \big)^{-L^{\prime}}.  
\end{align*}
For any given $L>0$, 
the above estimates with a 
sufficiently large $L^{\prime}$ imply  
\begin{equation}\label{I+j}
\big| I^{+}_{j} (x) \big| 
\lesssim 
(2^{j})^{\frac{1}{2}+\frac{n}{2}} 
\big( 1+ 2^j \big| 1-|x| \big| \big)^{-L}, 
\quad 
2^{j}|x| >1.  
\end{equation}

%%============================
Estimate of $I^{-}_{j}(x)$ for $2^{j}|x| >1$.  
The function $I^{-}_{j} (x)$ is written as 
\begin{equation*}
I^{-}_{j} (x)
=
 b^{-}\, 
\big\{ 
( |x| t )^{\frac{1}{2}-\frac{n}{2}}\, 
\psi (2^{-j}t )\, t^{n-1} 
\big\} ^{\wedge} (1+|x|). 
\end{equation*}
Hence, by the same reason as in the case of $I^{+}_{j}(x)$, 
\[
|I^{-}_{j} (x)|
\lesssim 
|x|^{\frac{1}{2}-\frac{n}{2}}\, 
(2^{j})^{\frac{1}{2}+\frac{n}{2}}\, 
\big( 2^{j} \big|1+|x|\big| \big)^{-L^{\prime}} 
\]
for any $L^{\prime}> 0$. 
Restricting to the region $2^{j}|x|>1$, we have 
\begin{equation}\label{I-j}
|I^{-}_{j} (x)|
\lesssim 
(2^{j})^{n-L^{\prime}}\, 
(1+|x|)^{-L^{\prime}},  
%\le 
%(2^{j})^{n-\frac{L^{\prime}}{2}}\, 
%\big( 2^{j} (1+|x| )\big)^{-\frac{L^{\prime}}{2}}, 
\quad 2^{j}|x|>1.  
\end{equation}

%%==================================
Estimate of $K^{+}_{j}(x)$ for $2^{j} |x|>1$. 
$K^{+}_{j} (x)$ is written as 
\begin{equation*}
K^{+}_{j} (x)
=
 \big\{ 
R^{+}(|x|t)\, 
\psi ( 2^{-j} t) \, t^{n-1} \big\}^{\wedge}
(1-|x|). 
\end{equation*}
The function $R^{+}(|x|t)\, 
\psi ( 2^{-j} t) \, t^{n-1} $ 
is supported on $\{2^{j-1}\le t \le 2^{j+1}\}$. 
If $2^j |x|>1$, then \eqref{remainder}
implies 
\[
\left| 
\partial_{t}^{\ell} 
\big\{ 
R^{+}(|x|t)\, 
\psi ( 2^{-j} t) \, t^{n-1}
\big\}
\right| 
\lesssim 
|x|^{\frac{1}{2}- \frac{n}{2} -1}
\big( 2^j \big)^{\frac{1}{2}+ \frac{n}{2}-2 -\ell}, 
\quad \ell = 0,1,2, \dots, 
\]
which, via Fourier transform, yields 
\[
%\big| K^{+}_{j} (x) \big| 
%=
\big| \big\{ 
R^{+}(|x|t)\, 
\psi ( 2^{-j} t) \, t^{n-1} \big\}^{\wedge}
(1-|x|)\big|
\lesssim 
|x|^{\frac{1}{2}-\frac{n}{2}-1} 
(2^{j})^{\frac{1}{2}+\frac{n}{2}-1} 
\big( 1+ 2^j \big| 1-|x| \big| \big)^{-L^{\prime}}. 
\]
for any $L^{\prime}> 0$. 
Hence 
\begin{align*}
&
2^{-j}< |x|\le 2^{-1} \; \Rightarrow \;
|K^{+}_{j} (x)|
\lesssim 
(2^{-j})^{^{\frac{1}{2}-\frac{n}{2}-1}}\, 
(2^{j})^{\frac{1}{2}+\frac{n}{2}-1}\, 
(2^{j})^{-L^{\prime}}
=
(2^j)^{n-L^{\prime}}, 
\\
&
|x|> 2^{-1}\; \Rightarrow \; 
|K^{+}_{j} (x)|
\lesssim 
(2^{j})^{\frac{1}{2}+\frac{n}{2}-1}\, 
\big( 1+ 2^{j} \big|1-|x|\big| \big)^{-L^{\prime}}.  
\end{align*}
For any $L>0$, the above estimates with $L^{\prime}$ sufficiently large 
implies 
\begin{equation}\label{K+j}
\big| K^{+}_{j} (x) \big| 
\lesssim 
(2^{j})^{\frac{1}{2}+\frac{n}{2}-1} 
\big( 1+ 2^j \big| 1-|x| \big| \big)^{-L}, 
\quad 
2^{j}|x| >1.  
\end{equation}

%%==================================
Estimate of $K^{-}_{j}(x)$ for $2^{j} |x|>1$. 
$K^{-}_{j} (x)$ is written as 
\begin{equation*}
K^{-}_{j} (x)
=
 \big\{ 
R^{-}(|x|t)\, 
\psi ( 2^{-j} t) \, t^{n-1} \big\}^{\wedge}
(1+|x|). 
\end{equation*}
If $2^j |x|>1$, then by the same reasoning as 
above we obtain 
\[
\big| \big\{ 
R^{-}(|x|t)\, 
\psi ( 2^{-j} t) \, t^{n-1} \big\}^{\wedge}
(1+|x|)\big|
\lesssim 
|x|^{\frac{1}{2}-\frac{n}{2}-1} 
(2^{j})^{\frac{1}{2}+\frac{n}{2}-1} 
\big( 2^j  (1+|x| ) \big)^{-L^{\prime}}
\]
for any $L^{\prime}> 0$. 
Hence 
\begin{equation}\label{K-j}
|K^{-}_{j} (x)|
\lesssim 
(2^{j})^{n-L^{\prime}}\, 
\big( 1+|x| \big)^{-L^{\prime}},  
\quad 2^{j}|x|>1.  
\end{equation}

Now from 
\eqref{I+j}, 
\eqref{I-j}, 
\eqref{K+j}, and 
\eqref{K-j}, 
we obtain 
the estimate of $h_j (x)$ for $2^j |x|>1$ as claimed in the lemma.  
Thus the claim (1) is proved.

{\it Proof of $(2)$.}\/
The equality \eqref{formula-I+j} and the equality 
$b^{+}= 
(2\pi)^{-\frac{n+1}{2}}
e^{-i  (\frac{(n-2)\pi}{4}+\frac{\pi}{4})}$ 
give 
\begin{equation}\label{I+0-psi}
e^{i  (\frac{(n-2)\pi}{4}+\frac{\pi}{4})}
(2^{j})^{-\frac{n}{2}-\frac{1}{2}} 
I^{+}_{j} (x) = 
 (2\pi)^{-\frac{n+1}{2}}  |x|^{\frac{1}{2}-\frac{n}{2}} 
\big( \psi (t)\, t^{-\frac{1}{2}+ \frac{n}{2}}
\big)^{\wedge} (2^{j} (1-|x|)) . 
\end{equation}
We 
set 
\[c_0 = 
(2\pi)^{-\frac{n+1}{2}} 
\big( \psi (t)\, t^{-\frac{1}{2}+ \frac{n}{2}}
\big)^{\wedge} (0). 
\]
This is a positive number since $\psi $ is nonnegative and not identically 
equal to $0$. 
Then, from \eqref{I+0-psi} and from continuity of the functions, 
it follows that there exists a number $\delta>0$ such that 
\[
1- 2^{-j} \delta < |x|< 1+ 2^{-j} \delta
\;\; \Rightarrow \;\; 
\left| 
e^{i  (\frac{(n-2)\pi}{4}+\frac{\pi}{4})}
(2^{j})^{-\frac{n}{2}-\frac{1}{2}} 
I^{+}_{j}(x) - c_0 \right|
\le \frac{c_0}{20}. 
\]
On the other hand, 
the estimates of \eqref{I-j}, \eqref{K+j}, and 
\eqref{K-j} 
imply that there exists a constant $c_1=c_1 (n, \psi)$ such that 
\[
1- 2^{-j}  < |x|< 1+ 2^{-j} 
\;\; \Rightarrow \;\; 
|I^{-}_{j}(x)|
+|K^{+}_{j}(x)|
+|K^{-}_{j}(x)|
\le 
c_{1}\,
2^{j(\frac{n}{2}+\frac{1}{2}-1)}. 
\]
Hence the estimate claimed in (2) of the lemma 
holds 
if we take $j_0$ large enough so that 
$c_{1}\, 2^{-j_0}\le c_0/20$. 
This completes the proof of 
Lemma \ref{lem-220919}. 
\end{proof}

\begin{proof}[Proof of Theorem \ref{th-E}.]
%{\it Proof of $(1)$.}\/ 
We define the operator $S_j$ by 
\[
S_j h = 
\left( 
e^{i |\xi|} 
\theta (2^{-j} \xi) 
\widehat{h}(\xi) 
\right)^{\vee}. 
\]
We divide the proof into three cases.

%%=============================
{\it Case 1:  $0<p,q\le2$.}\/ 
%%=============================
Assume \eqref{eq-nec2} holds, or equivalently 
\begin{equation}\label{eq-nec3}
2^{jm} \left\| 
S_j f \cdot 
S_j g  
\right\|_{X_r} 
\le A  
\|f\|_{H^p} 
\|g\|_{H^q}
\;\; 
\text{for all}
\;\; 
j \in \N. 
\end{equation}

Take $\psi$ as in Lemma \ref{lem-220919} and 
set 
\begin{equation*}
f_j (x)= \left( \psi (2^{-j} |\xi|) \right)^{\vee}(x), \quad j \in \N. 
\end{equation*}
%where $\xi, x \in \R^n$. 
We shall test \eqref{eq-nec3} to $f=g=f_j$.

Since the support of the Fourier transform of $f_j$ is included in 
the annulus $\{2^{j-1}\le |\xi|\le 2^{j+1}\}$ and since 
$f_j (x) = 2^{jn} \left( \psi ( |\cdot|) \right)^{\vee}(2^j x)$, it 
follows that 
\begin{equation*}
\|f_j \|_{H^p}
\approx 
\|f_j \|_{L^p}
\approx 
2^{j (n-\frac{n}{p})}
\end{equation*}
and similar estimate holds for $\|f_j \|_{H^q}$. 
On the other hand, 
by the choice of the functions $\theta$ and $\psi$, 
we have 
\begin{equation*}
S_j f_j = \left( e^{i |\xi|} 
\psi (2^{-j} |\xi|) \right)^{\vee}. 
\end{equation*}
Hence, by Lemma \ref{lem-220919}, 
there exist $\delta \in (0, \infty)$ and $j_0\in \N$ 
such that 
\begin{equation*}
\left|S_j f_j (x)\right|
\gtrsim 
2^{j \frac{n+1}{2}} 
\ichi \{2^{j} \left| 1-|x| \right|<\delta\}, 
\quad 
j>j_0. 
\end{equation*}
Thus 
\begin{equation*}
\left\| (S_j f_j)^{2} \right\|_{L^r}
\gtrsim 
\left( 2^{j \frac{n+1}{2}} \right)^{2}
\left(
\int 
\ichi \{2^{j} \left| 1-|x| \right|<\delta\}
\, dx
\right)^{1/r}
\approx 2^{j (n+1-\frac{1}{r})} 
\end{equation*} 
for $j>j_0$. 
Hence, if \eqref{eq-nec3} holds, then 
testing it to $f=g=f_j$ we have 
\[
2^{jm}\cdot 2^{j (n+1-\frac{1}{r})}
\lesssim 
2^{j (n-\frac{n}{p})} 
\cdot 
2^{j (n-\frac{n}{q})} 
\] 
for $j>j_0$, 
which is possible only when 
$m\le -(n-1)\big( \frac{1}{p} - \frac{1}{2} + \frac{1}{q} - \frac{1}{2}\big)$.

%%=============================
{\it Case 2:  $2\le p, q \le \infty$.}\/ 
%%=============================
Assume \eqref{eq-nec3} holds. 
Using the function $\psi$ of Lemma \ref{lem-220919}, 
we set 
\begin{equation*}
\widetilde{f}_j =  \left( e^{-i |\xi|} 
\psi (2^{-j} |\xi|) \right)^{\vee}, 
\quad j\in \N. 
\end{equation*}
Then 
Lemma \ref{lem-220919} gives the estimate 
\begin{equation*}
\left\|\widetilde{f}_j \right\|_{H^p}
\approx 
\left\|\widetilde{f}_j \right\|_{L^p}
\approx 
2^{j (\frac{n+1}{2}-\frac{1}{p})}
%\quad 
%\|f_j \|_{H^q}
%\approx 
%\|f_j \|_{L^q}
%\approx 
%2^{j (n-\frac{n}{q})}. 
\end{equation*}
and similar estimate holds for $\|\widetilde{f}_j \|_{H^q}$. 
On the other hand, 
\begin{equation*}
S_j \widetilde{f}_j (x) 
= 
\left( \psi (2^{-j} |\xi|)
\right)^{\vee} (x) 
= 
2^{jn} 
\left( \psi ( |\cdot |)
\right)^{\vee} (2^{j}x) 
\end{equation*}
and hence 
\begin{align*}
\left\| 
(S_j \widetilde{f}_j)^{2} 
\right\|_{X_r}
=
\left\| 
2^{2jn} 
\left( \psi ( |\cdot |)
\right)^{\vee} (2^{j}x)^{2} 
\right\|_{X_r}
\approx 
2^{j (2n - \frac{n}{r})}. 
%\quad 0<r\le \infty. 
\end{align*}
Hence, if \eqref{eq-nec3} holds, then 
by testing it to $f=g=\widetilde{f}_j$ we have 
\[
2^{jm}\cdot 2^{j (2n-\frac{n}{r})}
\lesssim 
2^{j (\frac{n+1}{2}-\frac{1}{p})} 
\cdot 
2^{j (\frac{n+1}{2}-\frac{1}{q})},  
\] 
which is possible only when 
$m\le -(n-1)( \frac{1}{2} - \frac{1}{p} + \frac{1}{2} - \frac{1}{q})$.

%%=============================
{\it Case 3:  $1\le p \le 2\le q \le \infty$ or 
$1\le q \le 2\le p \le \infty$ and $1/p+1/q=1$.}\/ 
%%=============================

By the symmetry of the situation, it is sufficient to 
consider the case $1\le p \le 2\le q \le \infty$. 
Thus we assume 
$1\le p \le 2\le q \le \infty$ and $1/p+1/q=1/r=1$. 
We assume 
\eqref{eq-nec2} holds, or equivalently, 
\begin{equation}\label{eq-nec4}
2^{jm} \left\| 
S_j f \cdot S_j g  
\right\|_{L^1} 
\le A  
\|f\|_{H^p} 
\|g\|_{L^q}
\;\; 
\text{for all}
\;\; 
j \in \N, 
\end{equation}
and prove that this is possible only when $m\le -n/p + n/2$.

We use the same function $f_j$ that was 
used 
in the proof of Case 1: 
\[
f_j (x)= \left(\psi (2^{-j} |\xi|)\right)^{\vee} (x), 
\quad j \in \N, 
\]
where $\psi$ is the function given in Lemma \ref{lem-220919}.

As we have seen in Case 1, 
\begin{equation}\label{eq-fjHp}
\|f_j \|_{H^p}
%\approx 
%\|f_j \|_{L^p}
\approx 
2^{j (n-\frac{n}{p})}. 
\end{equation}
On the other hand, 
\[
S_j f_j (x) = \left( e^{i |\xi|} \psi (2^{-j} |\xi|)\right)^{\vee}(x) 
=
\overline{\left( e^{-i |\xi|} \psi (2^{-j} |\xi|)\right)^{\vee}(-x) }
\]
and, hence, 
Lemma \ref{lem-220919} (2) gives 
\begin{equation}\label{Sjfjx}
\big|e^{-i(\frac{(n-2)\pi}{4}+\frac{\pi}{4})}\, 
2^{-j\frac{n+1}{2}} \, 
S_j f_j (x)
- c_0\, \big|
\le \frac{c_0}{10}
\;\; 
\text{if}
\;\; 
1- \delta 2^{-j } < |x|  <1+ \delta 2^{-j}  
\; \; \text{and}\;\; 
j > j_0. 
\end{equation}

For a sequence of complex numbers 
$\alpha = \big( \alpha_{\ell }\big)_{\ell \in \Z^n}$, 
we define 
$g_{j, \alpha}$ by 
\[
g_{j, \alpha} (x) 
= 
\sum_{\ell \in \Z^n}\, 
\alpha_{\ell}\, 
f_j (x - \delta^{\prime} 2^{-j} \ell), 
\]
where $\delta^{\prime}$ is a 
sufficiently small positive number; 
for the succeeding argument 
the choice 
$\delta^{\prime} = {\delta}/{(2 \sqrt{n})}$
will suffice.

We shall prove 
\begin{equation}\label{eq-gjLq}
\|g_{j, \alpha}\|_{L^q} \lesssim 2^{j(n-\frac{n}{q})} \|\alpha\|_{\ell^q}. 
\end{equation}
In fact, 
since 
$ 
f_j (x) 
%= \big( \psi (2^{-j} |\xi|)\big)^{\vee}(x) 
=
2^{jn} 
\big( \psi (|\cdot |)\big)^{\vee}(2^{j}x) 
$ 
and since 
$\big( \psi (|\cdot|)\big)^{\vee}$ is a 
Schwartz function, 
we have 
$
\big| f_j (x) \big|
\lesssim 
2^{jn} \big( 1 + 2^{j}|x| \big)^{-L}
$ 
for any $L>0$. 
Thus, if $2\le q <\infty$, then 
H\"older's inequality yields 
\begin{align*}
& |g_{j, \alpha} (x)|
\lesssim 
\sum_{\ell \in \Z^n} 
| \alpha_{\ell} | 
2^{jn} \big( 1 + 2^{j}|x - \delta^{\prime} 2^{-j} \ell| \big)^{-L}
\\
& 
\le 
\left( 
\sum_{\ell \in \Z^n} 
| \alpha_{\ell} |^q  2^{jnq} 
\big( 1 + 2^{j}|x - \delta^{\prime} 2^{-j} \ell| \big)^{-L}
\right)^{1/q}
\left( 
\sum_{\ell \in \Z^n} 
\big( 1 + 2^{j}|x - \delta^{\prime} 2^{-j} \ell| \big)^{-L}
\right)^{1-{1}/{q}}
\\
&
\approx 
\left( 
\sum_{\ell \in \Z^n} 
| \alpha_{\ell} |^q  2^{jnq} 
\big( 1 + 2^{j}|x - \delta^{\prime} 2^{-j} \ell| \big)^{-L}
\right)^{1/q}
\end{align*}
and hence 
\begin{align*}
\left\| g_{j, \alpha} \right\|_{L^q}
\lesssim 
\left( 
\int 
\sum_{\ell \in \Z^n} 
| \alpha_{\ell} |^q  2^{jnq} 
\big( 1 + 2^{j}|x - \delta^{\prime} 2^{-j} \ell| \big)^{-L}
\, dx 
\right)^{1/q}
\approx 
\|\alpha \|_{\ell^q} 2^{j (n-\frac{n}{q})}. 
\end{align*}
An obvious modification gives \eqref{eq-gjLq} for $q=\infty$ as well.

Since the operator $S_j$ is linear and commutes with translation, 
we have
\[
S_j g_{j, \alpha} =
\sum_{\ell \in \Z^n} 
\alpha_{\ell} 
(S_j f_j)
(x - \delta^{\prime} 2^{-j} \ell).  
\]

Now we test \eqref{eq-nec4} to 
$f=f_j$ and $g=g_{j, \alpha}$. 
Then by \eqref{eq-fjHp} and \eqref{eq-gjLq} we have 
\begin{equation}\label{B2jn-1}
2^{jm} 
\left\| 
S_j f_j (x)   
\sum_{\ell \in \Z^n} \alpha_{\ell} 
S_j f_j  (x- \delta^{\prime} 2^{-j} \ell ) 
\right\|_{L^1_{x}}
\lesssim 
%2^{j(n-\frac{n}{p})} \cdot 
%2^{j(n-\frac{n}{q})} \|\alpha\|_{\ell^q}
%=
2^{jn} \|\alpha\|_{\ell^q}
\end{equation}
(recall that $1/p+1/q=1$). 
We take the dual form of this inequality, 
which reads as 
\begin{equation}\label{B2jn-2}
2^{jm} 
\bigg\| 
\int 
S_j f_j (x)
S_j f_j  (x- \delta^{\prime} 2^{-j} \ell ) 
\varphi (x)\, dx \, 
\bigg\|_{\ell^{q^{\prime}}_{\ell}}
\lesssim 
2^{jn}\|\varphi\|_{L^{\infty}}. 
\end{equation}

We define the cube $Q_{\nu}$ in $\R^n$ by 
\[
Q_{\nu}
=\delta^{\prime} 2^{-j} \big( \nu + (0,1]^n\big), 
\quad \nu \in \Z^n.  
\]
Then each $Q_{\nu}$ is a cube with side length $\delta^{\prime} 2^{-j} $ 
and all of them constitute a partition of  $\R^n$. 
Let 
$\big(\epsilon_{\nu}\big)_{\nu \in \Z^n}$ be any sequence of 
$\pm 1$, 
and apply 
\eqref{B2jn-2} to 
$
\varphi (x) = \sum_{\nu \in\Z^n} \epsilon_{\nu} \ichi_{Q_{\nu}}(x)
$. 
Then 
we obtain 
\begin{equation*}
2^{jm} 
\left( 
\sum_{\ell \in \Z^n} 
\bigg| 
 \sum_{\nu \in\Z^n} \epsilon_{\nu} 
\int_{Q_{\nu}} 
S_j f_j (x)
S_j f_j  (x- \delta^{\prime} 2^{-j} \ell ) 
\, dx\,  
\bigg|^{q^{\prime}}
\right)^{1/q^{\prime}}
\lesssim 2^{jn} .
\end{equation*}
Notice that this inequality holds uniformly 
for all choices of 
$\epsilon_{\nu}=\pm 1$. 
We take the $q^{\prime}$-th power of 
the above inequality, 
take average over all choices of 
$\epsilon_{\nu}=\pm 1$, 
and use 
Kintchine's inequality; 
this yields 
\begin{equation}\label{B2jn-3}
\sum_{\ell \in \Z^n} 
\left(  
 \sum_{\nu \in\Z^n} 
 \bigg| 
 \int_{Q_{\nu}} 
S_j f_j (x)
S_j f_j  (x- \delta^{\prime} 2^{-j} \ell ) 
\, dx\,  
\bigg|^2 \, 
\right)^{q^{\prime}/2}
\lesssim 
2^{j(n-m)q^{\prime}}
\end{equation}

We shall estimate the left hand side of \eqref{B2jn-3} from below. 
For 
$v\in \R^n$, we define 
\[
\Sigma (v)=
\{x \in \R^n \mid 
|x|=|x - v |=1 
\}. 
\]
If 
$0<|v|<2$, then 
$\Sigma (v)$ is a $n-2$ dimensional sphere of radius 
$\sqrt{1- 4^{-1}|v|^2}$. 
Thus, in particular, 
if $0<|v|<1 $ and 
$\eta > 0$ is sufficiently small, 
then 
the $n$-dimensional Lebesgue measure of 
the $\eta$-neighborhood of $\Sigma (v)$ satisfies 
\begin{equation}\label{Sv}
\big| \, \text{the $\eta$-neighborhood of}\; \Sigma (v)\, \big|
\approx \eta^{2}.  
\end{equation}

Suppose $\ell \in \Z^n$ satisfies 
\begin{equation}\label{assumption-ell}
0<|\delta^{\prime}2^{-j}\ell|<1  
\end{equation}
and consider $\nu \in \Z^n$  
that satisfies  
\begin{equation}\label{assumption-nu}
\dist \big(Q_{\nu}, \Sigma (\delta^{\prime}2^{-j}\ell)\big)
<
\frac{\delta 2^{-j}}{2}. 
\end{equation}
Then, for each $x \in Q_{\nu}$, there exists 
an $x^{\prime}\in \Sigma (\delta^{\prime}2^{-j}\ell)$ such that 
\[
|x-x^{\prime}|<\diam Q_{\nu}+\frac{\delta 2^{-j}}{2}
=\delta 2^{-j}, 
\]
and, since this $x^{\prime}$ satisfies 
$|x^{\prime}|=|x^{\prime}- \delta^{\prime} 2^{-j}\ell|=1$, 
we have 
\[
1- \delta 2^{-j}< |x|< 
1+ \delta 2^{-j}
\;\;
\text{and}\;\; 
1- \delta 2^{-j}< |x-\delta^{\prime} 2^{-j}\ell|
< 
1+ \delta 2^{-j}. 
\]
Hence, by 
\eqref{Sjfjx}, we see that 
\begin{align*}
&
\big|e^{-i(\frac{(n-2)\pi}{4}+\frac{\pi}{4})}\, 
2^{-j\frac{n+1}{2}} \, 
S_j f_j (x)
- c_0\, \big|
\le \frac{c_0}{10}, 
\\
&
\big|e^{-i(\frac{(n-2)\pi}{4}+\frac{\pi}{4})}\, 
2^{-j\frac{n+1}{2}} \, 
S_j f_j (x-\delta^{\prime}2^{-j}\ell)
- c_0\, \big|
\le \frac{c_0}{10}
\end{align*}
for all $x \in Q_{\nu}$ and all $j>j_0$, 
which implies that 
\begin{equation}\label{intSjfjSjgjell}
\bigg| 
\int_{Q_{\nu}} 
S_j f_j (x)
S_j f_j (x-\delta^{\prime}2^{-j}\ell)\, 
dx \bigg|
\approx 
2^{j\frac{n+1}{2}}\cdot 2^{j\frac{n+1}{2}}\cdot 2^{-jn}
=2^{j}, 
\quad 
j>j_0. 
\end{equation}

All the cubes $Q_{\nu}$ that satisfy 
\eqref{assumption-nu} certainly 
cover the 
$\frac{1}{2}\delta 2^{-j}$-neighborhood of 
$\Sigma (\delta^{\prime} 2^{-j}\ell)$. 
Conversely, 
since 
$\diam Q_{\nu}=2^{-1}\delta 2^{-j}$, 
all $Q_{\nu}$ that satisfy \eqref{assumption-nu} 
are included in 
the 
$\delta 2^{-j}$-neighborhood of 
$\Sigma (\delta^{\prime} 2^{-j}\ell)$. 
Hence, by \eqref{Sv}, we see that 
\begin{equation}\label{card-nu}
\card \big\{ \nu \in \Z^n 
\mid 
\nu \;\; \text{satisfies}
\;\; 
\eqref{assumption-nu}
\big\} 
\approx 
\frac{2^{-2j }}{2^{-jn}}
=
2^{j (n-2)} 
\end{equation}
for each $\ell$ satisfying \eqref{assumption-ell}. 
Also we have obviously  
\begin{equation}\label{card-ell}
\card \big\{ \ell \in \Z^n 
\mid 
\ell \;\; \text{satisfies}
\;\; 
\eqref{assumption-ell}
\big\} 
\approx 
2^{j n}.  
\end{equation}

From 
\eqref{intSjfjSjgjell}, \eqref{card-nu}, and \eqref{card-ell},  
we see that 
\begin{align*}
&
\big(\, 
\text{the left hand side of} 
\; 
\eqref{B2jn-3}\, \big) 
\\
&
\ge 
\sum_{\ell\, : \, \eqref{assumption-ell}} 
\bigg(  
 \sum_{\nu\, : \, \eqref{assumption-nu}} 
 \bigg| 
 \int_{Q_{\nu}} 
S_j f_j (x)
S_j f_j  (x - \delta^{\prime}2^{-j}\ell)
\, dx 
\bigg|^2 \, 
\bigg)^{q^{\prime}/2}
\\
&
\approx  
\sum_{\ell\, : \, \eqref{assumption-ell}} 
\big(  
( 
2^{j}
)^2 
\cdot  
2^{j (n-2)}
\big)^{q^{\prime}/2}
%= 
%2^{j\frac{n}{2}} \, 
%\card \{\ell\in \Z^n \mid 
%0<|\delta^{\prime} 2^{-j}\ell|<1
%\}
%\\
%&
\approx 2^{j (nq^{\prime}/{2}+n)}
\end{align*}
for all $j>j_0$.

Thus \eqref{B2jn-3} implies 
$2^{j (nq^{\prime}/{2}+n)}\lesssim 
2^{j (n-m)q^{\prime}}$ for $j>j_0$, 
which is possible only when 
$m\le -n/2+n/q=n/2 -n/p$.  
This completes the proof of Theorem \ref{th-E}. 
\end{proof}

%%%%%%%%%%%%%%%%%%%%%%%%%%%%%%%%%
\section{Proofs of Propositions \ref{prop-a0a1a2} 
and \ref{prop-a0a1}} 
\label{G}
%%%%%%%%%%%%%%%%%%%%%%%%%%%%%%%%%

%%%%%%%%%%%%%%%%%%%%%%%%%%%%%%%%%
\subsection{Proof of Proposition \ref{prop-a0a1a2} } 
\label{G1}
%%%%%%%%%%%%%%%%%%%%%%%%%%%%%%%%%

In order to prove Proposition \ref{prop-a0a1a2}, we use the following lemmas. 
The first two lemmas are given in \cite{MT-flag}. 

%%%=========================================
\begin{lem}[{\cite[Lemma 2.5]{MT-flag}}]
\label{flag-Lem2.5}
Let $0<p,q \le \infty$ 
and $1/p+1/q=1/r>0$.
Assume that $\psi$ and $\phi$ are functions on 
$\R^n$ 
such that 
$\supp \psi \subset
\{a^{-1} \le |\xi| \le a\}$ and 
\begin{align*}
&
\left| 
\partial_{x}^{\alpha}
(\psi)^{\vee} (x)
\right| 
\le A (1+|x|)^{-L}, 
\quad 
|\alpha|=0, 1, 
\\
&
\left| 
\partial_{x}^{\beta}
(\phi)^{\vee} (x)
\right| 
\le B (1+|x|)^{-L}, 
\quad |\beta|\le L^{\prime}, 
\end{align*}
where $a, A, B\in (0,\infty)$ and 
$L$ and $L^{\prime}$ are sufficiently large 
integers determined by $p, q$, and $n$.   
Then 
\begin{equation*}
\left\|
\left(
\sum_{j \in \Z}
\left|\psi (2^{-j}D) f 
\cdot \varphi (2^{-j}D)g\right|^{2}
\right)^{1/2}
\right\|_{L^r}
\le c 
AB \|f\|_{H^p} \|g\|_{H^q}, 
\end{equation*}
where $c=c(n, p, q, a)$ is a positive constant. 
Moreover, if $p=\infty$ then $\|f\|_{H^p}$
can be replaced by $\|f\|_{BMO}$.
\end{lem}
%%%==============================

%%%================================
\begin{lem}[{\cite[Lemma 2.7]{MT-flag}}]
\label{flag-Lem2.7}
Let $0<p,q \le \infty$ 
and $1/p+1/q=1/r>0$.
Assume that $\psi_1$ and $\psi_2$
are functions on $\R^n$ such that
$\supp \psi_1 $, $\supp \psi_2 
\subset \{a^{-1}\le |\xi|\le a\}$ 
and 
\begin{align*}
&
\left| 
\partial_{x}^{\alpha}
(\psi_1)^{\vee} (x)
\right| 
\le A (1+|x|)^{-L}, 
\quad 
|\alpha|\le L^{\prime}, 
\\
&
\left| 
\partial_{x}^{\beta}
(\psi_2)^{\vee} (x)
\right| 
\le B (1+|x|)^{-L}, 
\quad |\beta|\le L^{\prime}, 
\end{align*}
where $a, A, B\in (0,\infty)$ and 
$L$ and $L^{\prime}$ are sufficiently large 
integers determined by $p, q$, and $n$.   
Then 
\[
\left\|
\sum_{j \in \Z}
\left|\psi_1 (2^{-j}D)f \cdot 
\psi_2 (2^{-j}D)g
\right|
\right\|_{L^r}
\le c 
AB 
\|f\|_{H^p} \|g\|_{H^q}, 
\end{equation*}
where $c=c(n, p, q, a)$ is a positive constant. 
Moreover, if $p=\infty$ (respectively, $q=\infty$) 
then $\|f\|_{H^p}$ (respectively, $\|g\|_{H^q}$)
can be replaced by $\|f\|_{BMO}$
(respectively, $\|g\|_{BMO}$).
\end{lem}
%%%=============================

%%============================
\begin{lem}\label{lem-paraproduct-A}
Let $m_2 < 0$ and suppose 
the multiplier $\tau$ is given by 
\begin{equation}
\tau (\xi, \eta)
=
\sum_{j-k \ge 3}
c_{j,k} 
\psi_{1}(2^{-j}\xi)
\psi_2 (2^{-k}\eta) ,  
\end{equation}
where 
$(c_{j,k})$ a sequence of complex numbers satisfying 
$\left| c_{j,k} \right| \le 2^{(j-k)m_2}$ and 
$\psi_1 , \psi_2$ are functions in $C_{0}^{\infty}(\R^n)$ 
such that 
$\supp \psi_1, \, \supp \psi_2 \subset 
\{2^{-1}\le |\xi|\le 2\}$. 
Then $\tau$ belongs to the following 
multiplier classes: 
\begin{align*}
&
\calM (H^p \times H^q \to L^r), 
\quad 
0<p,q<\infty, \;\; 1/p+1/q=1/r, 
\\
&
\calM (H^p \times BMO \to L^p), 
\quad 
0<p<\infty, 
\\
&
\calM (BMO \times H^q \to L^q), 
\quad 0<q<\infty, 
\\
&
\calM (BMO \times BMO \to BMO). 
\end{align*}
Moreover, in each case, 
the multiplier norm of $\tau$ 
is bounded by 
$c \|\psi_1\|_{C^N} 
\|\psi_2\|_{C^N}$ 
with 
$c=c(n,m_2, p,q)$ and $N=N(n,p,q)$. 
\end{lem}
%%============================

\begin{proof} 
We divide the proof into several cases. 

{$(1^{\circ})$ $H^p \times H^q \to L^r$, $0<p,q<\infty$, $1/p+1/q=1/r$.}  
From the assumption $|c_{j,k}|\le 2^{(j-k)m_2}$ with $m_2<0$, 
we can use Schur's lemma 
\begin{align*}
\left| 
T_{\tau}(f,g)(x)
\right|
&
=
\left| 
\sum_{j-k \ge 3}
c_{j,k} 
\psi_1 (2^{-j}D)f(x) 
\psi_2 (2^{-k}D)g(x)
\right|
\\
&
\le 
\sum_{j-k \ge 3}
2^{(j-k)m_2}
\left| \psi_1 (2^{-j}D)f(x)\right|\,  
\left| \psi_2 (2^{-k}D)g(x)\right|
\\
&
\lesssim 
\left\| \psi_1 (2^{-j}D)f(x)\right\|_{\ell^2_j}
\left\| \psi_2 (2^{-k}D)g(x)\right\|_{\ell^2_k}  
\end{align*}
(for Schur's lemma, 
see {\it e.g.}\ \cite[Appendix A]{G-modern}).  
The above inequality together with 
H\"older's inequality and the Littlewood--Paley inequalities 
gives 
\begin{align*}
\left\| T_{\tau} (f,g) \right\|_{L^r}
&
\lesssim 
\left\| \left\| \psi_1 (2^{-j}D)f(x)\right\|_{\ell^2_j} \right\|_{L^p_x}
\left\| \left\| \psi_2 (2^{-k}D)g(x)\right\|_{\ell^2_k} \right\|_{L^q_x}
\\
& 
\lesssim 
\|\psi_1\|_{C^N} 
\|f\|_{H^p} 
\|\psi_2\|_{C^N}\|g\|_{H^q}, 
\end{align*}
which is the desired estimate.

{$(2^{\circ})$ $H^p \times BMO \to L^p$, $0<p<\infty$.} 
Observe that, if $j-k\ge 3$, then 
the support of the Fourier transform 
of 
$\psi_1 (2^{-j}D)f\cdot \psi_2 (2^{-k}D)g$ is 
included in the annulus 
$\{ 2^{j-2} \le |\zeta| \le 2^{j+2}\}$. 
%\[
%\left\{
%\xi + \eta \mid 
%\xi \in \supp \psi_1 (2^{-j}\cdot ), 
%\quad 
%\eta \in \supp \psi_2 (2^{-k}\cdot ) 
%\right\}
%\subset 
%\left\{
%\zeta \mid 
%2^{j-2}\le |\zeta| \le 2^{j+2}
%\right\}
%\]
Hence, the Littlewood-Paley theory for $H^p$ gives
\begin{align*}
&
\left\| 
\sum_{j-k \ge 3}
c_{j,k}
\psi_1 (2^{-j}D)f\cdot 
\psi_2 (2^{-k}D)g
\right\|_{L^p}
\\
&
\lesssim 
\left\| 
\sum_{j-k \ge 3}
c_{j,k}
\psi_1 (2^{-j}D)f\cdot 
\psi_2 (2^{-k}D)g
\right\|_{H^p}
\\
&
\lesssim 
\left\|
\bigg\| 
\sum_{k=-\infty}^{j-3}
c_{j,k}
\psi_1 (2^{-j}D)f(x) 
\psi_2 (2^{-k}D)g(x)
\bigg\|_{\ell^2_j}
\right\|_{L^p_x}
\\
&=: (\ast). 
\end{align*}
Since 
$\| \psi_2 (2^{-k}D)g\|_{L^{\infty}}
\lesssim \|\psi_2\|_{C^N} \|g\|_{BMO}$ 
(see, {\it e.g.,}\/ \cite[Chapter IV, Section 4.3.3]{S} )
and since 
$\sum_{k=-\infty}^{j-3} |c_{j,k}| \le 
\sum_{k=-\infty}^{j-3} 2^{(j-k)m_2} \approx 1$, 
we obtain 
\begin{align*}
(\ast)
&
\lesssim 
\|\psi_2\|_{C^N} \|g\|_{BMO}
\left\| 
\big\| 
\psi_1 (2^{-j}D)f(x) 
\big\|_{\ell^2_j}
\right\|_{L^p_x}
\\
&
\lesssim 
\|\psi_2\|_{C^N} \|g\|_{BMO}
\|\psi_1\|_{C^N} \|f\|_{H^p}, 
\end{align*}
which is the desired estimate.

{$(3^{\circ})$ $BMO \times H^q \to L^q$, $1<q<\infty$.}
By the same reason as in $(2^{\circ})$, 
the Littlewood-Paley theory for $L^q$, $1<q<\infty$, yields 
\begin{align*}
&
\left\| 
\sum_{j-k \ge 3}
c_{j,k}
\psi_1 (2^{-j}D)f\cdot 
\psi_2 (2^{-k}D)g
\right\|_{L^q}
\\
&
\lesssim 
\left\|
\bigg\| 
\sum_{k=-\infty}^{j-3}
c_{j,k}
\psi_1 (2^{-j}D)f(x) 
\psi_2 (2^{-k}D)g(x)
\bigg\|_{\ell^2_j}
\right\|_{L^q_x}
\\
&=: (\ast \ast). 
\end{align*}
%\end{split}
%\end{equation}
Take a function $\theta \in C_{0}^{\infty}(\R^n)$ such 
that 
$\theta (\eta)=1$ 
for $|\eta|\le 2$. 
Then, for $j-k\ge 3$, we have 
\[
\psi_2 (2^{-k}D) g (x) 
= \theta (2^{-j} D) \psi_2 (2^{-k}D) g (x) 
= \int 2^{jn} (\theta)^{\vee} \left(2^j(x-y)\right) 
\psi_2 (2^{-k}D) g (y)\, dy.   
\]
Combining this formula with 
the inequality 
$|\psi_2 (2^{-k}D) g (y)| \lesssim \|\psi_2\|_{C^N} Mg (y)$, 
where $M$ is the Hardy-Littlewood maximal operator,  
and with the inequality 
$|(\theta)^{\vee} (z)|\lesssim (1+|z|^2)^{-L/2}$, 
we have  
\[
\left| \psi_2 (2^{-k}D) g (x) \right| 
\lesssim  
\| \psi_2 \|_{C^N} 
S_j (Mg) (x), 
\]
where $S_j$ is defined by 
\[
S_j h (x) 
=
\int 2^{jn} \big( 1+\big|2^j (x-y) \big|^2  \big)^{-L/2}\, h(y)\, dy
\]
with $L>0$ sufficiently large. 
Hence 
\begin{align*}
(\ast \ast)
&
\lesssim  
\left\|
\bigg\| 
\sum_{k=-\infty}^{j-3}
2^{(j-k)m_2}
\left| \psi_1 (2^{-j}D)f(x) \right|
\|\psi_2\|_{C^N} S_j (Mg) (x)
\bigg\|_{\ell^2_j}
\right\|_{L^q_x}
\\
&
\approx 
\|\psi_2\|_{C^N}
\left\|
\bigg\| 
\left| \psi_1 (2^{-j}D)f(x) \right|
S_j (Mg) (x)
\bigg\|_{\ell^2_j}
\right\|_{L^q_x}
\\
&
\lesssim 
\|\psi_2\|_{C^N}
\| \psi_1 \|_{C^N} 
\|f\|_{BMO}
\|Mg\|_{L^q}
\\
&
\approx  
\|\psi_2\|_{C^N}
\| \psi_1 \|_{C^N} 
\|f\|_{BMO}
\|g\|_{L^q}, 
\end{align*}
where the second $\lesssim$ follows from 
Lemma \ref{flag-Lem2.5} and 
the last $\approx$ holds because $q>1$. 
 
{$(4^{\circ})$ $BMO \times H^q \to L^q$, $0<q\le 1$.} 
By virtue of the atomic decomposition for $H^q$, 
it is sufficient to show the uniform estimate 
of $\|T_{\tau}(f,g)\|_{L^q}$ for all $H^q$-atoms $g$. 
By translation, it is sufficient to consider 
the $H^q$-atoms  
supported on balls centered at the origin. 
Thus we assume 
\[
\supp g \subset \{|x|\le r\}, 
\quad 
\|g\|_{L^{\infty}}\le r^{-n/q}, 
\quad 
\int g(x)x^{\alpha}\, dx =0 
\;\; \text{for}\;\; 
|\alpha|\le [n/q-n], 
\]
and we shall prove 
$\|T_{\tau}(f,g)\|_{L^q}\lesssim \|\psi_1\|_{C^N}
\| \psi_2 \|_{C^N} 
\|f\|_{BMO}$.

By the same reason as in $(2^{\circ})$, 
the Littlewood-Paley theory 
for $H^q$ reduces the proof to the estimate of 
\[
\left\|
\bigg\| 
\sum_{k=-\infty}^{j-3}
c_{j,k}
\psi_1 (2^{-j}D)f(x) 
\psi_2 (2^{-k}D)g(x)
\bigg\|_{\ell^2_j}
\right\|_{L^q_x}. 
\]

We first estimate the $L^q$ norm on $|x|\le 2r$. 
Using H\"older's inequality and using 
the result proved in $(3^{\circ})$ (with $q=2$), 
we have 
\begin{align*}
&
\left\|
\bigg\| 
\sum_{k=-\infty}^{j-3}
c_{j,k}
\psi_1 (2^{-j}D)f(x) 
\psi_2 (2^{-k}D)g(x)
\bigg\|_{\ell^2_j}
\right\|_{L^q(|x|\le 2r)}
\\
&
\lesssim 
r^{\frac{n}{q}-\frac{n}{2}} 
\left\|
\bigg\| 
\sum_{k=-\infty}^{j-3}
c_{j,k}
\psi_1 (2^{-j}D)f(x) 
\psi_2 (2^{-k}D)g(x)
\bigg\|_{\ell^2_j}
\right\|_{L^2(|x|\le 2r)}
\\
&
\lesssim 
r^{\frac{n}{q}-\frac{n}{2}} 
\|\psi_1\|_{C^N} 
\|\psi_2\|_{C^N} 
\|f\|_{BMO}
\|g\|_{L^2}
\\
&
\lesssim 
\|\psi_1\|_{C^N} 
\|\psi_2\|_{C^N} 
\|f\|_{BMO}. 
\end{align*}

Next, we estimate the $L^q$ norm on $|x|> 2r$. 
Using the inequality 
$\|\psi_1 (2^{-j}D)f(x) \|_{L^{\infty}} \lesssim 
\|\psi_1\|_{C^N} \|f\|_{BMO}$, we have
\begin{align}
&
\left\|
\bigg\| 
\sum_{k=-\infty}^{j-3}
c_{j,k}
\psi_1 (2^{-j}D)f(x) 
\psi_2 (2^{-k}D)g(x)
\bigg\|_{\ell^2_j}
\right\|_{L^q(|x|> 2r)}
\nonumber 
\\
&
\lesssim 
\|\psi_1\|_{C^N} \|f\|_{BMO}
\left\|
\bigg\| 
\sum_{k=-\infty}^{j-3}
2^{(j-k)m_2}
\left| \psi_2 (2^{-k}D)g(x) \right| 
\bigg\|_{\ell^2_j}
\right\|_{L^q(|x|> 2r)}
\nonumber 
\\
&
\le 
\|\psi_1\|_{C^N} \|f\|_{BMO}
\left\|
\sum_{k=-\infty}^{\infty}
\big\| 
2^{(j-k)m_2}
\big\|_{\ell^2 (j\ge k+3)}
\big| \psi_2 (2^{-k}D)g(x) \big| 
\right\|_{L^q(|x|> 2r)}
\nonumber 
\\
&
\approx  
\|\psi_1\|_{C^N} \|f\|_{BMO}
\left\|
\sum_{k=-\infty}^{\infty} 
\left| \psi_2 (2^{-k}D)g(x) \right| 
\right\|_{L^q(|x|> 2r)}
\nonumber 
\\
&
\le 
\|\psi_1\|_{C^N} \|f\|_{BMO}
\left\| 
\left\| \psi_2 (2^{-k}D)g(x) 
\right\|_{L^q(|x|> 2r)}
\right\|_{\ell^q_k}. 
\label{111}
\end{align}
To estimate the $L^q$-norm of the functions 
$\psi_2 (2^{-k}D)g(x) $ on $|x|>2r$, 
we write 
\[
\psi_2 (2^{-k}D)g(x) 
=
\int_{|y|\le r} 
2^{kn} (\psi_2)^{\vee} \left(2^k (x-y)\right) 
g(y)\, dy.
\]
Then using the size estimate of $g$ and the 
moment condition on $g$, we have 
\[
\left| 
\psi_2 (2^{-k}D)g(x) 
\right| 
\lesssim 
\|\psi_2\|_{C^N} 
2^{kn} \left( 1+ 2^{k} |x|\right)^{-L} \, 
r^{-\frac{n}{q} +n}
\min \left\{
1, \, \big( 2^k r \big)^{[\frac{n}{q}-n]+1}
\right\}, 
\quad 
|x|> 2r  
\]
(see \cite[inequalities (2.7) and (2.8)]{MT-flag}).  
Hence 
\begin{align}
&
\left\| \psi_2 (2^{-k}D)g(x) 
\right\|_{L^q(|x|> 2r)}
\nonumber 
\\
&
\lesssim 
\|\psi_2\|_{C^N} \, 
r^{-\frac{n}{q}  +n} \min 
\left\{
1, \, \big( 2^k r \big)^{[\frac{n}{q} -n]+1}
\right\} 
\left\| 
2^{kn} \big( 1+ 2^{k} |x|\big)^{-L} 
\right\|_{L^q (|x|>2r)}
\nonumber 
\\
&
\approx 
\|\psi_2\|_{C^N} 
\min 
\left\{
\big( 2^{k}r \big)^{-L+n}, 
\, \big( 2^k r \big)^{n-\frac{n}{q}+[\frac{n}{q}-n]+1}
\right\} . 
\label{222}
\end{align}
From \eqref{111} and \eqref{222}, 
we obtain 
\begin{align*}
&
\left\|
\bigg\| 
\sum_{k=-\infty}^{j-3}
c_{j,k}
\psi_1 (2^{-j}D)f(x) 
\psi_2 (2^{-k}D)g(x)
\bigg\|_{\ell^2_j}
\right\|_{L^q(|x|> 2r)}
\\
&
\lesssim 
\|\psi_1\|_{C^N} \|f\|_{BMO}
\|\psi_2\|_{C^N}
\left\| 
\min 
\left\{
\big( 2^{k}r \big)^{-L+n}, 
\, \big( 2^k r \big)^{n-\frac{n}{q}+[\frac{n}{q}-n]+1}
\right\} 
\right\|_{\ell^q_k}
\\
&\lesssim  
\|\psi_1\|_{C^N} \|f\|_{BMO}
\|\psi_2\|_{C^N}.  
\end{align*}

{$(5^{\circ})$ $BMO \times BMO \to BMO$.}
By virtue of the duality between $BMO$ and $H^1$, 
it is sufficient to show the following inequality: 
\begin{equation}\label{fgh}
\begin{split}
&
\left| 
\int 
\sum_{j-k\ge 3} 
c_{j,k} 
\psi_1 (2^{-j}D) f(x)\, 
\psi_2 (2^{-k}D) g(x)\, 
h(x)\, 
dx \, 
\right|
\\
&
\lesssim 
\| \psi_1 \|_{C^N} \|f\|_{BMO} 
\| \psi_2 \|_{C^N} \|g\|_{BMO} 
\|h\|_{H^1}. 
\end{split}
\end{equation}
Notice that if $j-k \ge 3$ then 
the support of the Fourier transform 
of  $\psi_1 (2^{-j}D) f 
\cdot 
\psi_2 (2^{-k}D) g$ is included in the annulus 
$\{2^{j-2}\le |\zeta|\le 2^{j+2}\}$. 
Thus, if we take a function $\widetilde{\psi}\in C_{0}^{\infty}(\R^n)$ 
such that 
$\supp \widetilde{\psi} \subset \{2^{-3}\le |\zeta|\le 2^{3}\}$ and 
$\widetilde{\psi} (\zeta)=1$ on $2^{-2}\le |\zeta|\le 2^{2}$, 
then the integral in \eqref{fgh} can be written as 
\begin{align*}
&
\int 
\sum_{j-k\ge 3} 
c_{j,k} 
\psi_1 (2^{-j}D) f(x)\, 
\psi_2 (2^{-k}D) g(x)\, 
h(x)\, 
dx 
\\
&
=
\int 
\sum_{j-k\ge 3} 
c_{j,k} 
\psi_1 (2^{-j}D) f(x)\, 
\psi_2 (2^{-k}D) g(x)\, 
\widetilde{\psi}(2^{-j}D)h(x)\, 
dx. 
\end{align*} 
Hence, using the estimate 
$\|\psi_2 (2^{-k}D) g\|_{L^{\infty}}\lesssim 
\|\psi_2\|_{C^N}\|g\|_{BMO}$ 
and the assumption $|c_{j,k}| \le 2^{(j-k)m_2}$, 
$m_2<0$, 
we have 
\begin{align*}
&
\text{
(the left hand side of \eqref{fgh})
}
\\
&
\lesssim 
\int 
\sum_{j-k\ge 3} 
2^{(j-k)m_2}
\big| \psi_1 (2^{-j}D) f(x) \big| \, 
\big| \psi_2 (2^{-k}D) g(x)\big| \, 
\big| \widetilde{\psi}(2^{-j}D)h(x)\big| \, 
dx
\\
&
\lesssim 
\|\psi_2\|_{C^N}\|g\|_{BMO}
\int 
\sum_{j=-\infty}^{\infty} 
\big| \psi_1 (2^{-j}D) f(x) \big| \, 
\big| \widetilde{\psi}(2^{-j}D)h(x)\big| \, 
dx
\\
&
\lesssim 
\|\psi_2\|_{C^N}\|g\|_{BMO}
\|\psi_1\|_{C^N}\|f\|_{BMO}
\|h\|_{H^1}, 
\end{align*}
where the last $\lesssim$ follows from Lemma \ref{flag-Lem2.7}. 
This completes the proof of Lemma \ref{lem-paraproduct-A}.  
\end{proof}

%%============================
\begin{lem}\label{lem-paraproduct-B}
Suppose the multiplier $\tau$ is defined by 
\begin{equation*}
\tau (\xi, \eta)
=
\sum_{j=-\infty}^{\infty}
c_{j} 
\psi_{1}(2^{-j}\xi)
\phi (2^{-j+3}\eta)  
\end{equation*}
with 
a sequence of complex numbers 
$(c_{j})$ satisfying 
$\left| c_{j} \right| \le 1$ and 
with 
$\psi_1 , \phi \in C_{0}^{\infty}(\R^n)$ 
such that 
$
\supp \psi_1 \subset 
\{2^{-1}\le |\xi|\le 2\}$ and 
$\supp \phi \subset \{|\eta|\le 2\}$. 
Then $\tau$ belongs to the following multiplier classes: 
\begin{align*}
&
{\calM (H^p \times H^q \to L^r), \quad  0<p, q<\infty, \;\;1/p+1/q=1/r,} 
\\
&
{\calM (H^p \times L^{\infty} \to L^p), \quad 0<p<\infty,} 
\\
&
{\calM (BMO \times H^q \to L^q), \quad 0<q<\infty,} 
\\
&
{\calM (BMO \times L^{\infty} \to BMO).} 
\end{align*}
Moreover, in each case, 
the multiplier norm of $\tau$ 
is bounded by 
$c \|\psi_1\|_{C^N} 
\|\phi\|_{C^N}$ 
with 
$c=c(n,p,q)$ and $N=N(n,p,q)$. 
\end{lem}
%%============================

\begin{proof} 
%{\it Case 2: $H^p \times H^q \to L^r$, $0<p,q <\infty$, $1/p+1/q=1/r$.}\/
%
%{\it Case 2: $H^p \times L^{\infty} \to L^p$, $0<p<\infty$.}\/
%
%
%{\it Case 3: $BMO \times H^q \to L^q$, $0<q<\infty$.}\/ 
From the assumptions on the supports of $\psi_1$ and $\phi$, 
it follows that 
the support of the Fourier transform of 
$\psi_{1}(2^{-j}D)f \cdot 
\phi (2^{-j+3}D) g  $ is included in the annulus 
$\{2^{j-2}\le |\zeta|\le 2^{j+2}\}$. 
Hence, for $0<r<\infty$,  
the Littlewood-Paley theory implies 
\begin{align*}
&
\left\| 
\sum_{j=-\infty}^{\infty}
c_{j}
\psi_1 (2^{-j}D)f\cdot 
\phi (2^{-j+3}D)g
\right\|_{L^r}
\\
&
\lesssim 
\left\| 
\sum_{j=-\infty}^{\infty}
c_{j}
\psi_1 (2^{-j}D)f\cdot 
\phi (2^{-j+3}D)g
\right\|_{H^r}
\\
&
\lesssim 
\left\|
\bigg\| 
c_{j}
\psi_1 (2^{-j}D)f\cdot 
\phi (2^{-j+3}D)g
\bigg\|_{\ell^2_j}
\right\|_{L^r_x}
\\
&=: (\ast). 
\end{align*}
By Lemma \ref{flag-Lem2.5}, we have 
\begin{equation*}
(\ast)
\lesssim 
\|\psi_1\|_{C^N} 
\|\phi \|_{C^N} 
\begin{cases}
{\|f\|_{H^p}\|g\|_{H^q}, } 
& {\text{if}\;\;0<p,q<\infty \;\;\text{and}\;\; 1/p+1/q=1/r,} 
\\
{\|f\|_{H^p}\|g\|_{L^{\infty}}, } 
& {\text{if}\;\;0<p<\infty\;\;\text{and}\;\; p=r,}  
\\
{\|f\|_{BMO}\|g\|_{H^q}, } 
& {\text{if}\;\;0<q<\infty \;\; \text{and}\;\; q=r.}
\end{cases}
\end{equation*}
These prove the claims for the former three multiplier classes.

We shall prove 
$\tau \in \calM (BMO \times L^{\infty} \to BMO)$. 
By the same argument as given in 
$(5^{\circ})$ of the Proof of Lemma \ref{lem-paraproduct-A}, 
it is sufficient to show the inequality 
\begin{equation}\label{dualestimate}
\begin{split}
&
\left| 
\int 
\sum_{j=-\infty}^{\infty} 
c_{j} 
\psi_1 (2^{-j}D) f(x)\, 
\phi (2^{-j+3}D) g(x)\, 
\widetilde{\psi} (2^{-j}D)h(x)\, 
dx\, 
\right|
%\\
%&
%\lesssim 
%\| \psi_2 \|_{C^N} \|g\|_{L^{\infty}}  
%\sum_{j=-\infty}^{\infty} 
%\big| \psi_1 (2^{-j}D) f(x)\big| \, 
%\big| \widetilde{\psi} (2^{-j}D)h(x)\big| \, 
%dx
\\
&
\lesssim 
\| \phi \|_{C^N} \|g\|_{L^{\infty}}  
\| \psi_1 \|_{C^N} \|f\|_{BMO} 
\|h\|_{H^1}, 
\end{split}
\end{equation}
where $\widetilde{\psi}$ 
is the same function as given there. 
In the present case, using 
the assumption $|c_j|\le 1$ and 
the inequality 
$\| \phi (2^{-j+3}D) g\| _{L^{\infty}} \lesssim 
\|\phi\|_{C^N} \|g\|_{L^{\infty}}$, 
we see that 
\begin{align*}
&
(\text{the left hand side of \eqref{dualestimate}})
\\
&\lesssim 
\| \phi \|_{C^N} \|g\|_{L^{\infty}}  
\int \sum_{j=-\infty}^{\infty} 
\big| \psi_1 (2^{-j}D) f(x)\big| \, 
\big| \widetilde{\psi} (2^{-j}D)h(x)\big| \, 
dx. 
%
%\\
%&
%\lesssim 
%\| \phi \|_{C^N} \|g\|_{L^{\infty}}  
%\| \psi_1 \|_{C^N} \|f\|_{BMO} 
%\|h\|_{H^1}, 
\end{align*}
Now 
\eqref{dualestimate} follows 
from 
Lemma \ref{flag-Lem2.7}. 
This completes the proof of Lemma \ref{lem-paraproduct-B}.  
\end{proof}

\begin{proof}[Proof of Proposition \ref{prop-a0a1a2}]
We use several well-known methods developed in the 
theory of bilinear Fourier multiplier operators. 
We first decompose $\sigma (\xi, \eta)$ by using 
the usual dyadic partition of unity. 
Let $\psi$, $\zeta$, and $\varphi$ be the functions 
as given in Notation \ref{notation}. 

We decompose $\sigma$ into three parts: 
\begin{align*}
\sigma (\xi, \eta) 
&=
\sum_{j\in \Z} 
\sum_{k\in \Z} 
\sigma (\xi, \eta) \psi (2^{-j}\xi) \psi (2^{-k}\eta)
\\
&
=\sum_{j-k\ge 3}
+ 
\sum_{|j-k|\le 2}
+
\sum_{j-k\le -3}
\\
&
=
\sigma_{\RomI} (\xi, \eta)
+
\sigma_{\II} (\xi, \eta)
+
\sigma_{\III} (\xi, \eta), 
\end{align*}
where 
$\sum_{j-k\ge 3}$, $\sum_{|j-k|\le 2}$, and 
$\sum_{j-k\le -3}$ denote the sums of 
$\sigma (\xi, \eta) \psi (2^{-j}\xi) \psi (2^{-k}\eta)$ 
over $j,k \in \Z$ 
that satisfy the designated restrictions. 
We shall consider each of $\sigma_{\RomI}$, 
$\sigma_{\II}$, and $\sigma_{\III}$.

$(1^{\circ})$ 
For the multiplier $\sigma_{\II}$, 
we shall prove the following:  
\begin{align*}
&
\sigma_{\II} \in 
\calM (H^p \times H^q \to L^r), 
\quad 
0<p,q<\infty, \;\; 1/p+1/q=1/r, 
\\
&
\sigma_{\II} \in 
\calM (H^p \times BMO \to L^p), 
\quad 
0<p<\infty, 
\\
&
\sigma_{\II} \in 
\calM (BMO \times H^q \to L^q), 
\quad 
0<q<\infty, 
\\
&
\sigma_{\II} \in 
\calM (BMO \times BMO \to BMO). 
\end{align*}

To prove this, observe that
$|\xi|\approx |\eta|\approx 2^j$ 
on the support of 
$\psi (2^{-j}\xi) \psi (2^{-k}\eta)$ with $|j-k|\le 2$. 
From this we see that 
$\sigma_{\II} \in \dotS ^{0}_{1,0}(\R^{2n})$. 
Hence Proposition \ref{prop-CM} implies that 
$\sigma_{\II}$ is a bilinear Fourier multiplier for the following spaces: 
\begin{align*}
&
H^p \times H^q \to L^r, 
\quad 0<p,q<\infty, 
\; \; 
1/p+1/q=1/r, 
\\
&
H^p \times L^\infty \to L^p, 
\quad 0<p<\infty, 
\\
&
L^\infty \times H^q \to L^q, 
\quad 0<q<\infty, 
\\
&
L^{\infty} \times L^{\infty} \to BMO. 
\end{align*}
We shall prove that 
the space $L^\infty$ in the above 
can be replaced by $BMO$.

We use the Fefferman--Stein decomposition of $BMO$, 
which asserts that every $g\in BMO \cap L^2$ can be written as 
\[
g= g_0 + \sum_{\ell =1}^{n} R_{\ell} g_{\ell}, 
\quad 
\sum_{\ell=0}^{n} \|g_{\ell}\|_{L^{\infty}} \approx \|g\|_{BMO}, 
\] 
where $R_{\ell}h = \left( -i |\xi|^{-1}\xi_{\ell} \widehat{h}(\xi)\right)^{\vee}$ 
is the Riesz transform. 
(If $g\in BMO \cap L^2$, then we can take 
$g_{\ell}\in L^{\infty} \cap L^2$ and 
the equality 
$g= g_0 + \sum_{\ell =1}^{n} R_{\ell} g_{\ell}$ holds 
without modulo constants; see \cite{M-BMO}.) 
Thus 
\[
T_{\sigma_{\II}}(f,g)
=
T_{\sigma_{\II}}(f,g_0)
+\sum_{\ell=1}^{n} 
T_{\sigma_{\II}}(f,R_{\ell} g_\ell)
=
T_{\sigma_{\II}}(f,g_0)
+\sum_{\ell=1}^{n} 
T_{\sigma_{\ell, \II}}(f,g_\ell),
\]
where 
\[
\sigma_{\ell}(\xi, \eta) 
=
\sigma (\xi, \eta) ( -i |\eta|^{-1}\eta_{\ell} )
=
a_0 (\xi, \eta) 
a_{1}(\xi)  
a_2 (\eta) ( -i |\eta|^{-1}\eta_{\ell} )
\]
and 
$\sigma_{\ell, \II}$ is defined in the same way as 
$\sigma \mapsto \sigma_{\II}$. 
Since 
the multiplier 
$a_2 (\eta) ( -i |\eta|^{-1}\eta_{\ell} )$ 
belongs to $\dotS ^{-m_2}_{1,0}(\R^n)$, 
we can apply the result 
$\sigma_{\II} \in \calM ( H^p \times L^{\infty} \to L^p)$ 
to $\sigma_{\ell, \II}$ to see that 
\[
\left\| 
T_{\sigma_{\II}}(f,g)
\right\|_{L^p}
\lesssim 
\|f\|_{H^p}
\sum_{\ell=0}^{n}
\|g_{\ell}\|_{L^{\infty}}
\approx 
\|f\|_{H^p}
\|g\|_{BMO}. 
\] 
Thus 
$\sigma_{\II} \in \calM ( H^p \times BMO \to L^p)$. 
The claims  
$\sigma_{\II}\in \calM (BMO \times H^q \to L^q)$ and 
$\sigma_{\II}\in \calM (BMO \times BMO \to BMO)$ are 
proved in the same way. 

%Notice that the above arguments for $\sigma_{\II}$ hold 
%for all $m_1, m_2 \in \R$ that are not necessarily nonnegative. 

$(2^{\circ})$  
For the multiplier $\sigma_{\RomI}$, 
we shall prove the following: 
\begin{align}
&
\sigma_{\RomI} \in 
\calM (H^p \times H^q \to L^r) 
\quad 
\text{if}\;\; 0<p,q<\infty, \;\; 1/p+1/q=1/r, 
\label{001}
\\
&
\sigma_{\RomI} \in 
\calM (H^p \times BMO \to L^p) 
\quad 
\text{if}\;\; m_2<0 \;\;\text{and}\;\;0<p<\infty, 
\label{002}
\\
&
\sigma_{\RomI} \in 
\calM (BMO \times H^q \to L^q) 
\quad 
\text{if}\;\; 0<q<\infty, 
\label{003}
\\
&
\sigma_{\RomI} \in 
\calM (BMO \times BMO \to BMO) 
\quad 
\text{if}\;\; m_2<0. 
\label{004}
\end{align}

{\it Proof of \eqref{001} in the case $m_2=0$.}\/  
We write  
$\sigma_{\RomI}(\xi, \eta) =
b(\xi, \eta) a_2 (\eta)$ 
with 
\begin{equation}\label{def-b}
\begin{split}
b(\xi, \eta) 
&= \sum_{j-k \ge 3}
a_0 (\xi, \eta) a_1 (\xi) \psi (2^{-j}\xi) \psi (2^{-k}\eta)
\\
&
=\sum_{j=-\infty}^{\infty}
a_0 (\xi, \eta) a_1 (\xi) \psi (2^{-j}\xi) \varphi (2^{-j+3}\eta). 
\end{split}
\end{equation}
Since $m_2=0$ and $m=m_1$ in the present case, 
we see that $b\in \dotS ^{0}_{1,0}(\R^{2n})$. 
Thus, 
Proposition \ref{prop-CM} implies that 
$b\in \calM( H^p \times H^q \to L^r)$. 
Also since $a_2 \in \dotS ^{0}_{1,0}(\R^n)$ in the present case, 
the classical multiplier theorem for linear operators 
implies $a_2 \in \calM (H^q \to H^q)$. 
Hence $\sigma_{\RomI}\in \calM( H^p \times H^q \to L^r)$. 

{\it Proof of \eqref{001} in the case $m_2< 0$.}\/  
Notice that 
$\sigma_{\RomI}$ is supported in $|\xi|\ge 2 |\eta|$ and 
satisfies 
\[
\left| 
\partial_{\xi}^{\alpha} \partial_{\eta}^{\beta}
\sigma_{\RomI}(\xi, \eta)
\right| 
\le 
C_{\alpha, \beta} 
\bigg( 
\frac{|\xi|}{|\eta|}
\bigg)^{m_2}
|\xi|^{-|\alpha|}
|\eta|^{-|\beta|}. 
\]
Since $m_2 < 0$, 
the theorem of Grafakos and Kalton 
\cite[Theorem 7.4]{GK2} implies 
$\sigma_{\RomI}\in \calM (H^p \times H^q \to L^r)$.

{\it Another proof of \eqref{001} in the case $m_2<0$.}\/  
Here we shall give a direct proof of \eqref{001} for the 
case 
$m_2 <0$, which uses only a classical method. 
%This method was already used in 
%Proof of Theorem \ref{th-D4} 
%and will also be used in the 
%succeeding argument. 

Take a function $\widetilde{\psi}\in C_{0}^{\infty}(\R^n)$ such 
that 
$\supp \widetilde{\psi} \subset \{3^{-1}\le |\xi| \le 3\}$ 
and $\widetilde{\psi} (\xi)=1$  
for $2^{-1}\le |\xi| \le 2$. 
Then 
\[
\sigma_{\RomI}(\xi, \eta)
=
\sum_{j-k \ge 3} 
\sigma (\xi, \eta) 
\widetilde{\psi}(2^{-j}\xi) 
\widetilde{\psi}(2^{-k}\eta) 
\psi (2^{-j}\xi) \psi (2^{-k}\eta). 
\]
Consider the function 
\[
\sigma (2^j \xi, 2^{k}\eta)
\widetilde{\psi}(\xi) 
\widetilde{\psi}(\eta)
=
a_0 (2^j \xi, 2^k \eta)
a_1 (2^j \xi) 
a_2 (2^k \eta) 
\widetilde{\psi}(\xi) 
\widetilde{\psi}(\eta)
\]
with 
$j-k\ge 3$. 
This function is supported in 
$\{3^{-1}\le |\xi|\le 3\}\times 
\{3^{-1}\le |\eta|\le 3\}$ and satisfies the estimate 
\begin{equation*}
\left|
\partial_{\xi}^{\alpha}
\partial_{\eta}^{\beta}
\big\{ 
\sigma (2^j \xi, 2^{k}\eta)\widetilde{\psi}(\xi) 
\widetilde{\psi}(\eta)
\big\} 
\right| 
\le 
C_{\alpha, \beta}\,  2^{(j-k)m_2}
\end{equation*}
with $C_{\alpha, \beta}$ independent of $j, k\in \Z$. 
Hence using the Fourier series expansion we can write 
\[
\sigma (2^j \xi, 2^{k}\eta)\widetilde{\psi}(\xi) 
\widetilde{\psi}(\eta)
=
\sum_{a,b \in \Z^n} 
c_{j,k}^{(a,b)} 
e^{i a \cdot \xi} e^{i b \cdot \eta}, 
\quad 
|\xi|<\pi, \;\; |\eta|<\pi,  
\]
with the coefficient satisfying 
\begin{equation}\label{F-cjab-decay}
\left| c_{j,k}^{(a,b)} \right| 
\lesssim 
2^{(j-k)m_2} 
(1+|a|)^{-L} (1+|b|)^{-L}
\end{equation}
for any $L>0$. 
Changing variables 
$\xi \to 2^{-j}\xi$ and $\eta \to 2^{-k}\eta$
and 
multiplying $\psi (2^{-j}\xi) \psi (2^{-k}\eta)$, 
we obtain 
\[
\sigma (\xi, \eta)
\psi (2^{-j}\xi) 
\psi (2^{-k}\eta)
=
\sum_{a,b \in \Z^n} 
c_{j,k}^{(a,b)} 
e^{i a \cdot 2^{-j} \xi} e^{i b \cdot 2^{-k}\eta}
\psi (2^{-j}\xi) 
\psi (2^{-k}\eta).  
\]
Thus $\sigma_{\RomI}$ is written as 
\begin{equation}\label{F-sigma-I}
\sigma_{\RomI} (\xi, \eta)
=
\sum_{a,b \in \Z^n}
%\sigma_{\RomI}^{(a,b)} (\xi, \eta), 
%\\
%&
%\sigma_{\RomI}^{(a,b)} (\xi, \eta)
%=
\sum_{j-k \ge 3}
c_{j,k}^{(a,b)} 
\psi^{(a)}(2^{-j}\xi)
\psi^{(b)}(2^{-k}\eta)  
\end{equation}
with
\begin{equation}\label{F-psia-psib}
\psi^{(a)}(\xi)
=e^{i a \cdot \xi} 
\psi (\xi), 
\quad 
\psi^{(b)}(\eta)
=e^{i b \cdot \eta} 
\psi (\eta).  
\end{equation}

Now applying Lemma \ref{lem-paraproduct-A} 
to $\psi_1 = \psi^{(a)}$ and 
$\psi_2 = \psi^{(b)}$, we obtain 
\begin{align*}
&
\left\| 
\sum_{j-k \ge 3}
c_{j,k}^{(a,b)} 
\psi^{(a)}(2^{-j}\xi)
\psi^{(b)}(2^{-j}\eta)
\right\|_{\calM (H^p \times H^q \to L^r)}
\\
&
\lesssim 
(1+|a|)^{-L} (1+|b|)^{-L} 
\|\psi^{(a)}\|_{C^N} 
\|\psi^{(b)}\|_{C^N} 
\\
&
\lesssim 
(1+|a|)^{-L+N} (1+|b|)^{-L+N}.  
\end{align*}
Taking $L$ sufficiently large and taking sum over 
$a,b \in \Z^n$, we obtain \eqref{001}. 

{\it Proof of \eqref{002}.}\/
Using 
\eqref{F-sigma-I}, 
\eqref{F-psia-psib}, and \eqref{F-cjab-decay}, 
we can derive \eqref{002} 
from Lemma \ref{lem-paraproduct-A}.

{\it Proof of \eqref{003}.}\/ 
If $m_2<0$, then by using 
\eqref{F-sigma-I}, 
\eqref{F-psia-psib}, and \eqref{F-cjab-decay}, 
we can derive 
\eqref{003} from 
Lemma \ref{lem-paraproduct-A}.

Assume $m_2=0$. 
Then we write $\sigma_{\RomI}$ as 
$\sigma_{\RomI}(\xi, \eta) =b(\xi, \eta) a_2 (\eta)$ with 
$b$ given by \eqref{def-b}. 
%
%\begin{align*}
%b(\xi, \eta) 
%&
%= 
%\sum_{j-k \ge 3}
%a_0 (\xi, \eta) a_1 (\xi) \psi (2^{-j}\xi) \psi (2^{-k}\eta)
%\\
%&
%=\sum_{j=-\infty}^{\infty}
%a_0 (\xi, \eta) a_1 (\xi) \psi (2^{-j}\xi) \varphi (2^{-j+3}\eta). 
%\end{align*}
Since 
$a_2 \in \dotS ^{0}_{1,0} (\R^n)$ 
in the present case ($m_2=0$),  
the linear multiplier theorem implies 
$a_2 \in \calM (H^q \to H^q)$. 
Hence \eqref{003} will follow if 
we prove 
$b\in \calM (BMO \times H^q \to L^q)$. 
By the same argument given in the proof of \eqref{001}, 
we can write $b$ as 
\begin{align}
&
b (\xi, \eta)
=
\sum_{a,b \in \Z^n}
\sum_{j=-\infty}^{\infty}
c_{j}^{(a,b)} 
\psi^{(a)}(2^{-j}\xi)
\varphi^{(b)}(2^{-j+3}\eta), 
\label{b-sumab-sumj}
\\
&
\big| c_{j}^{(a,b)} \big| 
\lesssim 
(1+|a|)^{-L} 
(1+|b|)^{-L}, 
\label{b-cj}
\\
&
\psi^{(a)}(\xi)
=e^{i a \cdot \xi} 
\psi (\xi), 
\quad 
\varphi^{(b)}(\eta)
=e^{i b \cdot \eta} 
\varphi (\eta).
\label{psia-varphib}
\end{align}
Now we apply Lemma \ref{lem-paraproduct-B} 
to $\psi_1 = \psi^{(a)}$ and 
$\phi = \varphi^{(b)}$ to obtain 
\begin{align*}
&
\left\| 
\sum_{j-k \ge 3}
c_{j,k}^{(a,b)} 
\psi^{(a)}(2^{-j}\xi)
\varphi^{(b)}(2^{-j+3}\eta)
\right\|_{\calM (BMO \times H^q \to L^q)}
\\
&
\lesssim 
(1+|a|)^{-L} (1+|b|)^{-L} 
\|\psi^{(a)}\|_{C^N} 
\|\psi^{(b)}\|_{C^N} 
\lesssim 
(1+|a|)^{-L+N} (1+|b|)^{-L+N}.   
\end{align*}
Taking $L$ sufficiently large and taking sum over 
$a,b \in \Z^n$, we 
obtain 
$b \in \calM (BMO \times H^q \to L^q)$.

{\it Proof of \eqref{004}.}\/ 
This is also derived from 
Lemma \ref{lem-paraproduct-A} 
by the use of 
\eqref{F-sigma-I}, 
\eqref{F-psia-psib}, 
and 
\eqref{F-cjab-decay}.  

%=====================
$(3^{\circ})$ For the multiplier $\sigma_{\III}$, 
the following hold: 
\begin{align*}
&
\sigma_{\III} \in 
\calM (H^p \times H^q \to L^r)
\quad 
\text{if}\;\; 0<p,q<\infty, \;\; 1/p+1/q=1/r, 
\\
&
\sigma_{\III} \in 
\calM (H^p \times BMO \to L^p)  
\quad 
\text{if}\;\; 0<p<\infty, 
\\
&
\sigma_{\III} \in 
\calM (BMO \times H^q \to L^q)  
\quad 
\text{if}\;\; m_1 < 0\;\;\text{and}\;\;0<q<\infty, 
\\
&
\sigma_{\III} \in 
\calM (BMO \times BMO \to BMO) 
\quad 
\text{if}\;\; m_1 < 0.  
\end{align*}
In fact, these follow from 
the results for $\sigma_{\RomI}$ by the obvious symmetry.

Now we obtain the conclusion of 
Proposition \ref{prop-a0a1a2} by combining  
the results of 
$(1^{\circ})$, $(2^{\circ})$, and $(3^{\circ})$. 
This completes the proof of Proposition \ref{prop-a0a1a2}. 
\end{proof}

%%%%%%%%%%%%%%%%%%%%%%%%%%%%%%%%%
\subsection{Proof of Proposition \ref{prop-a0a1} } 
\label{G2}
%%%%%%%%%%%%%%%%%%%%%%%%%%%%%%%%%

Let $\psi$ and $\varphi$ be the functions as given in 
Notation \ref{notation}. 
In the same way as in Proof of Proposition \ref{prop-a0a1a2}, 
we decompose $\tau$ into three parts: 
\begin{align*}
&
\tau (\xi, \eta) =
\tau_{\RomI} (\xi, \eta)
+
\tau_{\II} (\xi, \eta)
+
\tau_{\III} (\xi, \eta), 
\\
&
\tau_{\RomI} (\xi, \eta)
=
\sum_{j-k\ge 3} a_0 (\xi, \eta) a_1 (\xi) 
\psi (2^{-j}\xi) \psi (2^{-k}\eta), 
\\
&
\tau_{\II} (\xi, \eta)
=
\sum_{|j-k|\le 2} 
a_0 (\xi, \eta) a_1 (\xi) 
\psi (2^{-j}\xi) \psi (2^{-k}\eta), 
\\
&
\tau_{\III} (\xi, \eta)
=
\sum_{j-k\le -3} 
a_0 (\xi, \eta) a_1 (\xi) 
\psi (2^{-j}\xi) \psi (2^{-k}\eta). 
\end{align*}
We shall prove each of $ \tau_{\RomI}$, 
$\tau_{\II}$, and $\tau_{\III}$ belongs to 
the multiplier class as mentioned in the proposition.

{\it Proof of (1).}\/  
Let $0<p<\infty$. 
The multipliers $\tau_{\II}$ and $\tau_{\III}$ belong to 
$\calM (H^p \times BMO \to L^p)$. 
In fact, these are proved in 
$(1^{\circ})$ and $(3^{\circ})$ in Proof of Proposition \ref{prop-a0a1a2}.

We shall prove 
$\tau_{\RomI} 
\in \calM (H^p \times L^{\infty} \to L^p)$. 
%This multiplier is written as 
%\[
%\tau_{\RomI} (\xi, \eta) 
%=
%\sum_{j=-\infty}^{\infty} 
%a_0 (\xi, \eta) a_1 (\xi)
%\psi (2^{-j}\xi)
%\varphi (2^{-j+3} \eta). 
%\] 
By the same argument as in 
Proof of Proposition \ref{prop-a0a1a2} 
(see Proof of \eqref{003}), 
we can write $\tau_{\RomI}$ as 
\begin{equation}\label{tauI-paraproduct}
\tau_{\RomI} (\xi, \eta)
=
\sum_{a, b \in \Z^n} 
\sum_{j=-\infty}^\infty{}
c_{j}^{(a,b)}
\psi^{(a)}(2^{-j}\xi)
\varphi^{(b)} (2^{-j+3}\eta),  
\end{equation}
with $c_{j}^{(a,b)} $ satisfying \eqref{b-cj} and 
$\psi^{(a)}$ and $\varphi^{(b)}$ defined by 
\eqref{psia-varphib}. 
Then Lemma \ref{lem-paraproduct-B} gives 
\begin{align*}
&
\left\| \sum_{j=-\infty}^\infty{}
c_{j}^{(a,b)}
\psi^{(a)}(2^{-j}\xi)
\varphi^{(b)} (2^{-j+3}\eta)
\right\|_{\calM (H^p \times L^{\infty} \to L^p)}
\\
&
\lesssim 
(1+|a|)^{-L} (1+|b|)^{-L} 
\|\psi^{(a)}\|_{C^N} 
\|\varphi^{(b)}\|_{C^N} 
\lesssim 
(1+|a|)^{-L+N} (1+|b|)^{-L+N}. 
\end{align*}
Taking $L$ sufficiently large and taking sum over $a,b \in \Z^n$, 
we obtain $\tau_{\RomI} 
\in \calM (H^p \times L^{\infty} \to L^p)$. 
Thus the part (1) is proved.

{\it Proof of (2).}\/  
Here we assume $m_1 <0$. 
By the results proved in 
$(1^{\circ})$ and $(3^{\circ})$ in Proof of Proposition \ref{prop-a0a1a2},  
the multipliers 
$\tau_{\II}$ and $\tau_{\III}$ belong to 
$\calM (BMO \times BMO \to BMO)$. 
Recall that the multiplier $\tau_{\RomI}$ is written as 
\eqref{tauI-paraproduct} 
with $c_{j}^{(a,b)} $ satisfying \eqref{b-cj} and 
$\psi^{(a)}$ and $\varphi^{(b)}$ defined by 
\eqref{psia-varphib}. 
Hence we can prove 
$\tau_{\RomI} \in \calM (BMO \times L^{\infty} \to BMO)$ 
by using Lemma \ref{lem-paraproduct-B}. 
Thus the part (2) of Proposition \ref{prop-a0a1} 
is proved. 
This completes the proof of Proposition \ref{prop-a0a1}.

%%%%%%%%%%%%%%%%%%%%%%%%%%%%%%%%%%%%

%%%%%%%%%%%%%%%%%%%%%%%%%%%%%%%%%

\begin{thebibliography}{20}

\bibitem[BBMNT]{BBMNT}
\'A. B\'enyi, F. Bernicot, D. Maldonado, V. Naibo and R. Torres,
{On the H\"ormander classes of bilinear pseudodifferential operators II},
Indiana Univ. Math. J. 62 (2013), 1733--1764.


\bibitem[CM1]{CM1}
R. Coifman and Y. Meyer,
{Au del\`a des op\'erateurs pseudo-diff\'erentiels},
Ast\'erisque 57 (1978), 1--185.



\bibitem[CM2]{CM2}
R. Coifman and Y. Meyer,
{Nonlinear harmonic analysis, operator theory and P.D.E.},
in Beijing lectures in harmonic analysis (Beijing, 1984),
3--45, Ann. of Math. Stud. 112, Princeton Univ. Press, Princeton, NJ,
1986.



\bibitem[G1]{G-classical}
L. Grafakos,
{\it Classical Fourier Analysis\/}, 
3rd edition, GTM 249, Springer, New York, 2014.

\bibitem[G2]{G-modern}
L. Grafakos,
{\it Modern Fourier Analysis\/}, 
3rd edition, GTM 250, Springer, New York, 2014.



\bibitem[GK1]{GK1}
L. Grafakos and N. Kalton,
{Multilinear Calder\'on-Zygmund operators on Hardy spaces},
Collect. Math. 52 (2001), 169--179.

\bibitem[GK2]{GK2}
L. Grafakos and N. J. Kalton, 
The Marcinkiewicz multiplier condition 
for bilinear operators, 
Studia Math.\ {\bf 146} (2001), 115--156.

\bibitem[GT]{GT}
L. Grafakos and R. Torres,
{Multilinear Calder\'on-Zygmund theory},
Adv. Math. 165 (2002), 124--164.


\bibitem[CM3]{CM3}
Y. Meyer and R. Coifman,
{Wavelets: Calder\'on-Zygmund and multilinear operators},
Cambridge Stud. Adv. Math. 48,
Cambridge University Press, Cambridge, 1997.

\bibitem[KS]{KS}
C. Kenig and E. M. Stein, 
{Multilinear estimates and fractional integration}, 
Math. Res. Lett. 6 (1999), 1--15.

\bibitem[M1]{M-wave}
A. Miyachi, 
On some estimates 
for wave equations in $L^p$ and in $H^p$, 
J.\ Fac.\ Sci. Univ.\ Tokyo, Sect.\ IA Math.\ {\bf 27} 
(1980), 331--354.


\bibitem[M2]{M-BMO}
A. Miyachi, 
Some Littlewood-Paley type inequalities 
and their application to 
the Fefferman-Stein decomposition of BMO, 
Indiana Univ. Math. J. 
{\bf 39} (1990), 563--583.

\bibitem[MT1]{Miyachi-Tomita-IUMJ}
A. Miyachi and N. Tomita,
{Calder\'on-Vaillancourt type theorem for bilinear operators},
Indiana Univ. Math. J. 62 (2013), 1165--1201.


\bibitem[MT2]{MT-flag}
A. Miyachi and N. Tomita, 
Estimates for trilinear flag paraproducts on $L^{\infty}$ and 
Hardy spaces, 
Math. Z. 
{\bf 282} (2016), 577--613.


\bibitem[Mu1]{Muscalu1}
C. Muscalu,
Paraproducts with flag singularities I. A case study,
Rev. Mat. Iberoam. {\bf 23} (2007), 705--742.

\bibitem[Mu2]{Muscalu2}
C. Muscalu,
Flag paraproducts,
in Harmonic Analysis and Partial Differential Equations,
Contemp. Math. 505 (2010), 131--151.

\bibitem[MuS]{MuscaluSchlag}
C. Muscalu and W. Schlag,
{\it Classical and Multilinear Harmonic Analysis},\,
Vol. II,
Cambridge Univ. Press,
Cambridge, 2013.


\bibitem[RRS]{RRS}
S. Rodr\'iguez-L\'opez, 
D. Rule, and W. Staubach,
A Seeger-Sogge-Stein theorem for bilinear Fourier 
integral operators, 
Adv. Math. {\bf 264} (2014), 1--54. 


\bibitem[RS]{RS}
M. Ruzhansky and M. Sugimoto, 
A local-to-global boundedness argument and Fourier integral
operators, 
J. Math. Anal. Appl. 
{\bf 473} (2019), 892--904. 

\bibitem[SSS]{SSS}
A. Seeger, C. D. Sogge, and E. M. Stein, 
Regularity properties of Fourier integral operators, 
Ann. of Math. {\bf 134} (1991), 231--251.

\bibitem[S]{S}
E. M. Stein, 
{\it Harmonic Analysis, 
Real-Variable Methods, 
Orthogonality, 
and Oscillating Integrals}, 
Princeton Univ.\ Press, 
Princeton, NJ, 1993.

\bibitem[SW]{SW}
E. M. Stein and G. Weiss, 
{\it Introduction to Fourier Analysis on Euclidean 
Spaces}, 
Princeton Univ.\ Press, 
Princeton, NJ, 1971.
\end{thebibliography}
\end{document}